\documentclass[a4paper,10pt,reqno]{amsart}
\usepackage{xcolor}
\usepackage{tabularx}
\usepackage{ltablex}
\usepackage{caption}
\usepackage{microtype}
\usepackage{tikz}
\usetikzlibrary{cd}
\usetikzlibrary{shapes.geometric}
\usetikzlibrary {shapes.misc}
\usetikzlibrary{positioning}
\usetikzlibrary{patterns}
\usetikzlibrary{calc}
\usepackage{contour}
  \contourlength{2pt}
  \contournumber{120}

\usepackage[hypertexnames=false,linktoc=all]{hyperref}

\usepackage[poorman]{cleveref}
\crefname{thm}{Theorem}{Theorems}
\crefname{dfn}{Definition}{Definitions}
\crefname{lem}{Lemma}{Lemmas}
\crefname{pro}{Proposition}{Propositions}
\crefname{cor}{Corollary}{Corollaries}
\crefname{ex}{Example}{Examples}
\crefname{rem}{Remark}{Remarks}
\crefname{prob}{Problem}{Problems}
\crefname{conj}{Conjecture}{Conjectures}
\crefname{equation}{equation}{equations}
\crefname{conv}{Convention}{Conventions}
\crefname{section}{\S\!}{\S\S\!}
\crefname{subsection}{\S\!}{\S\S\!}
\crefname{Question}{Question}{Questions}
\newcommand{\triv}{\mathrm{triv}}
\newcommand{\hilb}{\text{-Hilb}}
\newcommand{\Per}{\mathrm{Per}}
\newcommand{\ch}{\mathrm{ch}}
\newcommand{\wt}{\widetilde}
\newcommand{\wh}{\widehat}

\newcommand{\ol}{\overline}
\newcommand{\NS}{\mathrm{NS}}
\newcommand{\Stab}{\mathrm{Stab}}
\newcommand{\Irr}{\mathrm{Irr}}
\usepackage{color}

\usepackage{comment}
\usepackage{enumerate}
\usepackage{enumitem}
\usepackage{latexsym}
\usepackage{amscd}
\usepackage{amsmath}
\usepackage{amssymb}
\usepackage{mathrsfs}

\usepackage{array}
\usepackage{amsthm}
\usepackage{float}
\usepackage{graphicx}
\usepackage{subfigure}
\theoremstyle{plain}
\newtheorem{theorem}{Theorem}[section]
\newtheorem{corollary}[theorem]{Corollary}
\newtheorem{lemma}[theorem]{Lemma}
\newtheorem{definition-lemma}[theorem]{Definition-Lemma}

\newtheorem{proposition}[theorem]{Proposition}
\newtheorem{conjecture}[theorem]{Conjecture}

\newtheorem{maintheorem}{Theorem}

\theoremstyle{definition}
\newtheorem{definition}[theorem]{Definition}
\newtheorem{remark}[theorem]{Remark}

\newtheorem{condition}[theorem]{Condition}
\newtheorem{example}[theorem]{Example}

\numberwithin{equation}{section}
\numberwithin{figure}{section}
\numberwithin{table}{section}

\usepackage{prettyref}

\newcommand{\pref}{\prettyref}

\newrefformat{th}{Theorem~\ref{#1}}
\newrefformat{cr}{Corollary~\ref{#1}}
\newrefformat{lm}{Lemma~\ref{#1}}
\newrefformat{dl}{Definition-Lemma~\ref{#1}}
\newrefformat{cl}{Claim~\ref{#1}}
\newrefformat{sl}{Sublemma~\ref{#1}}
\newrefformat{pr}{Proposition~\ref{#1}}
\newrefformat{cj}{Conjecture~\ref{#1}}
\newrefformat{st}{Step~\ref{#1}}
\newrefformat{sc}{Section~\ref{#1}}
\newrefformat{df}{Definition~\ref{#1}}
\newrefformat{rm}{Remark~\ref{#1}}
\newrefformat{q}{Question~\ref{#1}}
\newrefformat{pb}{Problem~\ref{#1}}
\newrefformat{cd}{Condition~\ref{#1}}
\newrefformat{eg}{Example~\ref{#1}}
\newrefformat{he}{Heore~\ref{#1}}
\newrefformat{fg}{Figure~\ref{#1}}
\newrefformat{tb}{Table~\ref{#1}}

\newrefformat{mthm}{Theorem~\ref{#1}}

\newdimen\argwidth
\def\db[#1\db]{%
 \setbox0=\hbox{$#1$}\argwidth=\wd0
 \setbox0=\hbox{$\left[\box0\right]$}
  \advance\argwidth by -\wd0
 \left[\kern.3\argwidth\box0 \kern.3\argwidth\right]}

\newcommand{\sgn}{\operatorname{sgn}}
\renewcommand{\Re}{\ensuremath{\mathop{\mathfrak{Re}}\nolimits}}
\renewcommand{\Im}{\ensuremath{\mathop{\mathfrak{Im}}\nolimits}}
\newcommand{\Res}{\operatorname{Res}}

\newcommand{\Ker}{\operatorname{Ker}}

\newcommand{\id}{\operatorname{id}}

\newcommand{\GL}{GL}

\newcommand{\SL}{SL}

\newcommand{\coh}{\operatorname{coh}}

\newcommand{\Pic}{\operatorname{Pic}}
\newcommand{\Ext}{\operatorname{Ext}}

\newcommand{\Hom}{\operatorname{Hom}}

\newcommand{\Amp}{\operatorname{Amp}}

\newcommand{\orb}{{\mathrm{orb}}}

\newcommand{\bA}{\ensuremath{\mathbb{A}}}

\newcommand{\bC}{\ensuremath{\mathbb{C}}}

\newcommand{\bH}{\ensuremath{\mathbb{H}}}

\newcommand{\bP}{\ensuremath{\mathbb{P}}}
\newcommand{\bQ}{\ensuremath{\mathbb{Q}}}
\newcommand{\bR}{\ensuremath{\mathbb{R}}}

\newcommand{\bZ}{\ensuremath{\mathbb{Z}}}

\newcommand{\scA}{\ensuremath{\mathcal{A}}}
\newcommand{\scB}{\ensuremath{\mathcal{B}}}
\newcommand{\scC}{\ensuremath{\mathcal{C}}}
\newcommand{\scD}{\ensuremath{\mathcal{D}}}

\newcommand{\scF}{\ensuremath{\mathcal{F}}}

\newcommand{\scH}{\ensuremath{\mathcal{H}}}

\newcommand{\scM}{\ensuremath{\mathcal{M}}}

\newcommand{\scO}{\ensuremath{\mathcal{O}}}
\newcommand{\scP}{\ensuremath{\mathcal{P}}}

\newcommand{\scS}{\ensuremath{\mathcal{S}}}
\newcommand{\scT}{\ensuremath{\mathcal{T}}}
\newcommand{\scU}{\ensuremath{\mathcal{U}}}

\newcommand{\scW}{\ensuremath{\mathcal{W}}}

\newcommand{\scY}{\ensuremath{\mathcal{Y}}}
\newcommand{\scZ}{\ensuremath{\mathcal{Z}}}

\newcommand{\bfL}{\ensuremath{\mathbf{L}}}

\newcommand{\bfR}{\ensuremath{\mathbf{R}}}

\newcommand{\Dtilde}{{\widetilde{D}}}

\makeatletter
\newcommand{\xleftrightarrows}[2][]{\mathrel{
 \raise.40ex\hbox{$
       \ext@arrow 3095\leftarrowfill@{\phantom{#1}}{#2}$}
 \setbox0=\hbox{$\ext@arrow 0359\rightarrowfill@{#1}{\phantom{#2}}$}
 \kern-\wd0 \lower.40ex\box0}}

\newcommand{\xrightleftarrows}[2][]{\mathrel{
 \raise.40ex\hbox{$\ext@arrow 3095\rightarrowfill@{\phantom{#1}}{#2}$}
 \setbox0=\hbox{$\ext@arrow 0359\leftarrowfill@{#1}{\phantom{#2}}$}
 \kern-\wd0 \lower.40ex\box0}}
\newcommand{\xleftrightarrow}[2][]{
     \ext@arrow 0055{\leftrightarrowfill@}{#1}{#2}
}
\def\leftrightarrowfill@{
 \arrowfill@\leftarrow\relbar\rightarrow
}  
\makeatother

\newcommand{\la}{\left\langle}
\newcommand{\ra}{\right\rangle}

\title[Ishii's conjecture and Bridgeland stability]{Ishii's conjecture and Bridgeland stability conditions for dihedral reflection groups}

\author[S. Nimura]{Shu Nimura}
\address{Graduate School of Mathematics, Nagoya University, Furocho, Chikusa-ku, Nagoya, 464-8602, Japan}
\email{shu.nimura.c6@math.nagoya-u.ac.jp}

\begin{document}
\begin{abstract}
We provide a new proof of Ishii's conjecture for any dihedral reflection group $G\subset \GL_2(\bC)$ from the viewpoint of Bridgeland stability conditions.
Our strategy is to reduce the problem, via the derived McKay correspondence, to a geometric construction of Bridgeland stability conditions on the root stack of the maximal resolution along the strict transform of the discriminant divisor.
\end{abstract}

\maketitle

%
%
%
%\tableofcontents

%
%------------------------------------------------------
%
\section{Introduction}\label{sc:Introduction}
\subsection{Motivation}
One of the central themes in the McKay correspondence is the relationship between birational geometry and moduli spaces of representations of noncommutative algebras.
Craw and Ishii proposed a conjecture \cite{Craw-Ishii} on the crepant resolutions of $\bC^3/G$ for any finite subgroup $G\subset \SL_3(\bC)$: every crepant resolution arises as a moduli space of $\theta$-stable representations by varying King stability $\theta\in \Theta(G)$ \cite{King}. The conjecture was proved for abelian groups in $\SL_3(\bC)$ by \cite{Craw-Ishii} and for general finite groups in $\SL_3(\bC)$ by \cite{Yamagishi}. Furthermore, inspired by the DK hypothesis, Ishii proposed the following conjecture in dimension two, extending the Craw--Ishii conjecture to the non-crepant setting:
\begin{conjecture}[\cite{Ishii_maximal}]
    Let $G\subset \GL_2(\bC)$ be any finite subgroup. Then, a resolution of $\bC^2/G$ is dominated by the maximal resolution of the pair $(\bC^2/G, B)$ if and only if $X\cong M_{\theta}$ for some generic $\theta\in \Theta(G)$.
\end{conjecture}
\noindent
For the definitions, see \S 2. Here, we only note that if $B=0$ (equivalently, if $G$ acts freely in codimension one), the crepant resolution, when it exists, is both minimal and maximal. Ishii proved the ``if'' part in full generality, and also the ``only if'' part for arbitrary abelian subgroups and for small finite subgroups $G\subset \GL_2(\bC)$, i.e., when $G$ acts freely in codimension one. Later, \cite{capellan_dihedral} proved the conjecture for dihedral reflection groups. This paper aims to give a new proof of Ishii's conjecture in the case of dihedral reflection groups via Bridgeland stability conditions and wall-crossing phenomena.

\subsection{Bridgeland stability}
Bridgeland stability conditions were introduced by Bridgeland \cite{Bridgeland_Stab} inspired by Douglas's work on $\Pi$-stability \cite{Douglas}. Bridgeland stability conditions play a central role in the study of triangulated categories arising in birational geometry, representation theory, enumerative geometry, and mirror symmetry. Bridgeland stability can be regarded as a generalization of classical notions of stability, such as Gieseker stability for sheaves on a variety and King stability for representations of a noncommutative algebra.  

In the context of birational geometry of surfaces, Toda interpreted the surface MMP in terms of wall-crossing phenomena for moduli spaces of Bridgeland stable objects \cite{Toda_surfaceMMP}. We briefly review Toda's work.  Let $Y$ be a smooth projective surface. The statement is formulated in the `real part' of $\Stab(Y)_\bR$ of the stability manifold $\Stab(Y)$. This space fits into the following cartesian diagram (see also \S 2)
\[
\begin{tikzcd}
    \Stab(Y)_\bR\arrow[r]\arrow[d,"\Pi_\bR"]&\Stab(Y)\arrow[d,"\Pi"]\\
\NS(Y)_\bR\arrow[r]&K^{\mathrm{num}}(Y)_\bC^\vee.
\end{tikzcd}\]
Here, we fix the notation: for a Bridgeland stability condition  $\sigma=(Z,\scA)$, $M_{\sigma}([\scO_y])$ denotes the algebraic space parametrizing $Z$-stable objects $E\in\scA$ with $[E]=[\scO_y]$ in $K^{\mathrm{num}}(Y)$ for $y\in Y$.
Toda proved the following theorem:
\begin{theorem}[\cite{Toda_surfaceMMP}]
    Let $Y$ be any smooth projective surface. For every smooth contraction $Y\to X$, we can associate a connected open subset $U(X)\subset \Stab(Y)_{\bR}$ such that 
\begin{enumerate}
    \item For any $\sigma\in U(X)$, $M_{\sigma}([\scO_x])\cong X$.
    \item If $X,X'$ are related by a single blowup, then $\ol{U(X)}\cap \ol{U(X')}$ is nonempty and has real codimension one.
\end{enumerate}
\end{theorem}
To prove this, Toda constructs a local section of $\Pi_\bR$
\[\mu'_{\mathrm{geom}}: A^{\mathrm{tot}}:=\bigcup_{Y\to X} \ol{ A^{\dagger}}(X)\to \Stab_\bR(Y),\]
 where the union runs over all smooth contractions
$Y\to X$, and $\ol{ A^{\dagger}}(X)$ is a subset of the Neron-Severi group $\NS(Y)_{\bR}$. $U(X)$ is defined to be the image of the interior $A^{\dagger}(X)$.
 
On the other hand, in the context of representation theory of algebras, Bayer--Craw--Zhang constructed a continuous map \cite{Bayer-Craw-Zhang}, 
\[\mu'_{\mathrm{alg}}:\Theta(G)\times \Lambda_1\to \Stab([\bC^2/G]),\]
where $\Lambda_1$ is a certain space of additional parameters. Via $\mu'_{\mathrm{alg}}$, one can identify King stability and Bridgeland stability, allowing one to interpret Ishii's conjecture in terms of Bridgeland moduli (see \S 4.1 for the details.)

\subsection{Main results}
Our strategy for Ishii's conjecture is to construct an analogous map to $\mu'_{\mathrm{geom}}$ for a certain DM stack $\scY$ and to compare $\mu'_{\mathrm{geom}}$ and $\mu'_{\mathrm{alg}}$ via the McKay correspondence.

To state our results more precisely, we prepare some notations. Let $G$ be a dihedral reflection group and set $H:=G\cap \SL_2(\bC)$. Let $Y_H:=H\hilb(\bC^2)$ be the $H$-Hilbert scheme defined by \cite{IN_first}. $Y_H$ admits a natural action of $G/H\cong \bZ_2$, and we consider $\scY:=[Y_H/\bZ_2]$, which can be regarded as the second root stack over the quotient variety $Y:=Y_H/\bZ_2$ along the branch locus. $Y$ is the maximal resolution in Ishii's conjecture. We write $\pi:\scY\to Y$ for the coarse moduli map, or the quotient map.  In this situation, the universal family of $H\hilb(\bC^2)$ with $G/H$-action induces an exact equivalence
\[\Phi:D_c(\scY)\cong D_c([\bC^2/G]),\]
where $D_c(-)$ denotes the bounded derived category of coherent sheaves whose cohomology sheaves are compactly supported. This provides the natural setting for Bridgeland stability for our purposes since the corresponding Bridgeland moduli spaces are well-behaved (unlike the derived category $D_Z(-)$ supported on a fixed compact subscheme $Z$). 

To formulate the analogue of Toda's work, we introduce $\NS_{\orb}(\scY)_{\bR}$, the orbifold N\'{e}ron-Severi group (see \S 3.3 for the definition). For our purposes, we consider 
\[\NS_{\orb}(\scY)_\bC:=\NS_{\orb}(\scY)_\bR\otimes \bC.\]
 We define $\Stab_{\mathrm{n}}(\scY)$ to be the space of normalized stability conditions, which fits into the following cartesian diagram (see \S 3.3)
\[
\begin{tikzcd}
    \Stab_{\mathrm{n}}(\scY)\arrow[r]\arrow[d,"\Pi_{\mathrm{n}}"]&\Stab(\scY)\arrow[d,"\Pi"]\\
\NS_{\orb}(\scY)_\bC\arrow[r]&K^{\mathrm{num}}_c(\scY)_\bC^\vee.
\end{tikzcd}
\]
As in the ordinary surface case, we consider the numerical class of skyscraper sheaves on $\scY$, which is explicitly described as $[\pi^*\scO_y]\in K_c^{\mathrm{num}}(\scY)$ for any $y\in Y$.
\begin{maintheorem}\label{mthm:A}
    For any smooth contraction $Y\to X$ of the maximal resolution $Y$, we can associate a connected open subset $\scU(X)\subset \Stab_{\mathrm{n}}(\scY)$ such that 
    \begin{enumerate}
        \item For any $\sigma\in \scU(X)$, $M_{\sigma}([\pi^*\scO_y])\cong X$.
        \item If $X$ and $X'$ are related by a single blowup, then $\ol{\scU(X)}\cap \ol{\scU(X')}$ is nonempty and has real codimension one.
    \end{enumerate}
\end{maintheorem}
As a result, the birational geometry of the maximal resolution $Y$ is encoded by wall-crossing phenomena of Bridgeland moduli spaces. More precisely, we will construct a continuous section of $\Pi_{\mathrm{n}}$
\[\mu_{\mathrm{geom}}:A^{\mathrm{tot}}\times \scB:=\left(\bigcup_{Y\to X} \ol{ A^{\dagger}}(X)\right)\times \scB\to \Stab_{\mathrm{n}}(\scY)\]
where $\ol{A^{\dagger}(X)}\subset \NS_{\orb}(\scY)_\bR$ is the counterpart of $\ol{A^{\dagger}(X)}$ appearing in Toda's construction, and $\scB$ is a small open neighborhood of $0$ in $\NS_{\orb}(\scY)_\bR$ that is independent of the contractions. We can think of $\mu_{\mathrm{geom}}$ as a small deformation of $\mu'_{\mathrm{geom}}$ inside $\Stab_{\mathrm{n}}(\scY)$. 
A subtle but crucial point is that the real slice $\Stab_\bR(\scY)$ is too small to compare $\mu'_{\mathrm{geom}}$ with $\mu'_{\mathrm{alg}}$. The two local sections do not admit a nonempty common domain inside $\Stab_\bR(\scY)$. The deformation by the additional parameter space $\scB$ is introduced to produce such a common domain. The main point is that, after the deformation, the two local sections glue along the resulting overlap:
\begin{maintheorem}\label{mthm:B}
    After identifying the two stability manifolds $\Stab(\scY)$ and $\Stab([\bC^2/G])$ via the derived equivalence $\Phi$, an explicit change of parameters, and a rotation action, the local sections of $\Pi_{\mathrm{n}}$
\[
\mu_{\mathrm{geom}}
\quad\text{and}\quad
\mu'_{\mathrm{alg}}
\]
coincide on their common domain.
\end{maintheorem}
See \pref{thm:two_slices_glue} for the precise statement. Combining \pref{mthm:A} and \pref{mthm:B}, we obtain the following corollary (\pref{cr:wall&chamber_induced}, \pref{cr:verify_conj}):
\begin{corollary}
    For each contraction $Y\to X$, we can associate a GIT chamber $V(X)$ in $\Theta(G)$ such that 
    \begin{enumerate}
        \item For any $\theta\in V(X)$, $M_\theta\cong X$.
        \item If $X$ and $X'$ are related by a single blowup, then $\ol{V(X)}\cap\ol{V(X')}$ is nonempty and has real codimension one.
    \end{enumerate}
\end{corollary}
In particular, we recover Capellan's theorem \cite{capellan_dihedral}
that Ishii's conjecture holds for dihedral reflection groups. An advantage of our approach is that the description of the Bridgeland moduli spaces is very simple: for a contraction $f:\scY\to X$, we construct the moduli space $\{\bfL f^*\scO_x\mid x\in X\}$. The main difficulty is to realize it as a Bridgeland moduli space and to identify it with a King moduli space. This suggests that Bridgeland stability provides a new perspective on Ishii's conjecture in a more general setting.

\subsection{Relations with existing works}
First, we discuss the relation with Ishii's approach. Our approach to Ishii's conjecture is analogous to his. His essential case is when $G\supset \{\pm 1\}=:N$. He considered the crepant resolution $Y_N \to \bC^2/N$, equipped with an action of $\ol{G}$. Given a smooth contraction $Y\to X$, he proved that $X$ can be realized as a moduli space of $\ol{G}$-constellations on $Y_N$. In other words, he realized $X$ as the moduli space of objects in $\coh_c[Y_N/\ol{G}] \subset D_c([Y_N/\ol{G}])$, which is derived equivalent to $D_c([\bC^2/G])$. 
On the other hand, our first aim is to realize $X$ as the moduli space of complexes in $D_c(\scY)$, which is derived equivalent to $D_c([\bC^2/G])$. To put it simply, we aim to modify Ishii's approach to a more flexible framework, namely, Bridgeland moduli spaces.
Ishii's proof cannot be applied directly to our case. Indeed, Ishii uses the fact that minimal resolution dominates $Y_N/\ol{G}$, which is not true in our situation since our $\bC^2/G$ is already the `minimal resolution'. 

Next, we discuss the relation with existing works on Bridgeland stability. While Bridgeland stability conditions have been extensively studied for many classes of smooth projective surfaces and certain threefolds, relatively few works treat stability conditions on $D_c(X)$ for non-projective varieties $X$; see, for example, \cite{Bayer-Craw-Zhang}. Our formulation mainly follows the framework developed in \cite{Bayer-Craw-Zhang}. 

On the other hand, there are several works on stability conditions for Deligne--Mumford stacks, including projective Kleinian orbisurfaces whose coarse moduli spaces are projective surfaces with only Kleinian singularities \cite{Lim-Rota_Stab_orbisurf}. Stability conditions have also been studied for root stacks over projective curves \cite{rota_thesis} and over projective surfaces \cite{liu-Shen_root_stack_stab}. We note that \cite{liu-Shen_root_stack_stab} mainly constructs stability conditions corresponding to the large volume limit. Although the class of stacks considered here is rather special, our setting admits a rich family of stability conditions. Some wall-crossing phenomena appearing in our construction (cf. \pref{lem:tilt_of_standard_heart}, \pref{lem:smallcomparison_stability}) already appeared implicitly in the one-dimensional projective settings \cite{rota_thesis}. Other wall-crossing phenomena can be regarded as direct generalizations of Toda's work \cite{Toda_surfaceMMP}. 

It is also worth mentioning that induced stability provides a useful method for studying Bridgeland stability conditions on $G$-equivariant derived categories, as developed in \cite{Macri-Mehrotra-Stellari} \cite{Polishchuk_constant-family}. Such induced stability conditions are often easier to analyze. This viewpoint plays a key role in the comparison between the two constructions of stability conditions carried out in \S 4.

\subsection{Plan of the paper}

In \S 2, we review the precise statement of Ishii's conjecture and several related notions. We also recall some preliminary results on Bridgeland stability conditions, especially for (non-projective) surfaces. 
In \S 3, we construct the map $\mu_{\mathrm{geom}}$ using Toda's construction. We also prove \pref{mthm:A} except for the proof of the existence of the Bridgeland moduli spaces. 
In \S 4, we prove \pref{mthm:B}. Using the gluing result established there, we also resolve the existence problem of the moduli spaces, thereby completing the proof of \pref{mthm:A}. 
Finally, in \S 5, for the reader's convenience, we sketch a proof of the support property for stability conditions on non-projective surfaces.

\subsection{Notation}

Throughout this paper, all schemes, DM stacks, and algebras are defined over $\bC$. We consider the left action of $G$ on $\bC^2$ induced by the embedding $G\subset \GL_2(\bC)$.
We frequently use $D_c(X)$ to denote the full triangulated subcategory of $D(X):=D^b(\coh X)$ consisting of objects whose cohomology sheaves are compactly supported. There is a natural equivalence $D^b(\coh_c X)\cong D_c(X)$ (for example, see \cite[Proposition 1.7.11]{Sheaves_on_manifolds}). 
For subcategories $\scA,\scB$ of a triangulated category $\scC$, we use the notation $\scA * \scB$ to denote the full subcategory consisting of objects $E$ fitting into an exact triangle $A \to E \to B \to A[1]$ for some $A\in \scA$ and $B\in \scB$.

\section*{Acknowledgements}

The author would like to express his sincere gratitude to his supervisor, Akira Ishii, for introducing the author to his conjecture and for his continuous encouragement. The author is also grateful to Riku Fushimi for valuable discussions.
This work was financially supported by JST SPRING, Grant Number
JPMJSP2125. The author would like to take this opportunity to thank the
“THERS Make New Standards Program for the Next Generation Researchers.”

\section{Preliminaries}\label{sec:pre}
\subsection{Representations of dihedral groups}\label{sc:rep}
For $n\geq 1$, we consider
\[
G:=\left\langle\alpha,\beta\right\rangle \subset \GL(2, \bC)
\]
where \[\alpha:=\begin{pmatrix} \zeta_{n} & 0 \\ 0 & \zeta_{n}^{-1} \end{pmatrix},\,\, \beta:=\begin{pmatrix} 0 & 1 \\ 1 &0 \end{pmatrix},\,\,\zeta_{n}:=\exp{\frac{2\pi\sqrt{-1}}{n}}. \] Put $H:=G\cap \SL(2, \bC)=\left\langle\alpha\right\rangle$. Then $G/H=\langle \beta\rangle\cong \bZ_2$. We write the representation of $H$ by
\[\rho_i: H\to \bC^*, \alpha\mapsto \zeta_{n}^i.\]
We denote the set of irreducible representations modulo isomorphism by $\Irr H, \Irr G$.
There is an action of $G/H$ on $\Irr H$ by
\[\beta\cdot \rho(h):=\rho(\beta^{-1}h\beta).\]
Then the action fixes $\rho_0$ if $n$ is odd or $\rho_0, \rho_{\frac{n}{2}}$ if $n$ is even. Any other irreducible representation $\rho_i$
belongs to the free orbit $\{\rho_i, \rho_{n-i}\}$. Let $\sgn$ denote the nontrivial irreducible representation of $G/H$ and regard it as an irreducible representation of $G$. By the Clifford theory, the irreducible representations of $G$ are given by
\[\Irr G=\begin{cases}
    \{\rho_0^+,\rho_0^-, \rho_i\oplus\rho_{n-i}\mid i=1,\dots,\frac{n-1}{2}\} &\text{if $n$: odd} \\
    \{\rho_0^+,\rho_0^-,\rho_{\frac{n}{2}}^+,\rho_{\frac{n}{2}}^-, \rho_i\oplus\rho_{n-i}\mid i=1,\dots,\frac{n}{2}-1  \}&\text{if $n$: even}
\end{cases}\]
where $\rho_i^+:=\rho_i,\,\,\rho_i^-:=\rho_i^+\otimes \sgn$ for $i=0,\frac{n}{2}$. 

\subsection{Moduli space of G-constellations}
For a finite subgroup $G\subset \GL_n(\bC)$, $G$-Hilbert scheme \cite{Nakamura} \cite{IN_first} is defined to be the moduli space of \emph{$G$-clusters}, where a $G$-cluster is a $G$-invariant 0-dimensional subscheme $Z$ of $\bC^n$ such that $H^0(\scO_Z)$ is isomorphic to the regular representation $\bC[G]$. As a generalization of $G$-clusters, Craw-Ishii introduced the notion of \emph{$G$-constellations} \cite{Craw-Ishii}. A $G$-constellation is a 0-dimensional $G$-equivariant sheaf $E$ on $\bC^n$ such that $H^0(E)\cong \bC[G]$. We write
\[
v:=[\bC[G]\otimes \scO_0]
\in K_c^{\mathrm{num}}([\bC^2/G])
\] (see also \pref{sc:Generalities_Stab}). For any $\theta\in \Theta(G):=v^\perp\subset \Hom(K_c^{\mathrm{num}}([\bC^2/G]),\bR)$, the notion of $\theta$-stability for $G$-constellation is defined. We note that any $\theta$-stable $G$-constellation is $\theta$-semistable by definition.  We say that $\theta\in \Theta(G)$ is \emph{generic} if any $\theta$-semistable $G$-constellation is $\theta$-stable. We consider the moduli functor defined as follows:
\[\scM_\theta(v)(S):=\{\text{flat families of $\theta$-stable $G$-constellations on $\bC^n$}\}/\sim\]
Here, a \emph{flat family of $\theta$-stable $G$-constellations} is a $G$-equivariant sheaf $E_S$ on $S\times \bC^n$ flat over $S$ such that each fiber $E_s$ is a $\theta$-stable $G$-constellation for every closed point $s\in S$. For two flat families $E_S, E'_S$, we write $E_S\sim E'_S$ if there is $L\in \Pic(S)$ such that $E_S\cong E'_S\otimes q^*L$. Using King's construction of the moduli space of representations of a quiver \cite{King}, the moduli functor $\scM_\theta(v)$ is represented by a quasi-projective scheme $M_\theta(v)$ \cite{Craw-Ishii}. 

We also note that the set of generic King stability parameters forms a dense open subset of $\Theta(G)$ consisting of finitely many open convex polyhedral cones (e.g. \cite[Lemma 3.1]{Craw-Ishii}). We call this decomposition the GIT wall-and-chamber structure.

\subsection{Maximal resolutions and Ishii's  conjecture}\label{subsec:max_resol}
We briefly review Ishii's conjecture and the geometry of maximal resolutions for dihedral reflection groups.

We first recall the definition of maximal resolution. Let $(X, B)$ be a klt pair. A resolution $f: Y\to X$ is a \emph{maximal resolution} if the exceptional divisors of $f$ are precisely the prime divisors over $X$ with non-positive discrepancies. The maximal resolution exists uniquely in dimension two as in \cite{KSB}. In the case of $X=\bC^2/G$, we define the boundary divisor $B$ by the equality $\nu^*(K_{\bC^2/G}+B)=K_{\bC^2}$ for $\nu: \bC^2\to \bC^2/G$. If $G$ is a dihedral reflection group, we can write $B=\frac{1}{2}D$ by the discriminant divisor $D$. Now, we restate the conjecture by Ishii
\begin{conjecture}[\cite{Ishii_maximal}]
    Let $G\subset \GL_2(\bC)$ be a finite subgroup and consider the
quotient $\bC^2/G$ with the boundary divisor $B$ defined as above. For any resolution of singularities $f: X \to \bC^2/G$, there is a generic $\theta\in \Theta(G)$  such that $X\cong M_{\theta}(v)$ if and only if there is a morphism $Y \to X\to \bC^2/G$.
Here, $Y$ is the maximal resolution of $(\bC^2/G,B)$.
\end{conjecture}
Ishii's conjecture for dihedral reflection groups was proved by
Capellan \cite{capellan_dihedral} by using the variation of GIT for
$G$-constellations on $\bC^3$ (see \pref{rm:capellan_approach}). We next review the geometry of maximal
resolutions in this case, following \cite{IshiiNimurareflection}.
Let $Y_H$ be the minimal resolution of the quotient singularity $\bA^2/H$ of type $A_{n-1}$, whose exceptional locus consists of $n-1$ exceptional curves $E_1, \dots, E_{n-1}$ such that
$E_{i-1} \cap E_i$ consists of a point. 
Connected components of the fixed point locus of the action of $G/H$ on $Y_H$
are smooth divisors as in \cite[Lemma 2.1(1)]{IshiiNimurareflection}.
Moreover, the action of $G/H$ exchanges the strict transforms in $Y_H$ of the images of the two coordinate axes of $\bA^2$, one intersecting $E_1$
and the other intersecting $E_{n-1}$.
Then, it also exchanges  $E_i$ with $E_{n-i}$,
and no exceptional curve $E_i$ is contained in the fixed point set $R:=(Y_H)^{G/H}$.
Therefore, $R$ coincides with the strict transform $\Dtilde \subset Y$ of the discriminant divisor $D \subset X$.
It follows that $Y:=Y_H/(G/H)$ is smooth and is obtained by blowing up $\bA^2/G\cong\bA^2$ along $\lfloor \frac{n}{2} \rfloor$ points. As in \cite[Lemma 2.1]{IshiiNimurareflection}, $Y$ is the maximal resolution of $(\bC^2/G,\frac{1}{2}D)$ (see also \cite[Theorem 1.2]{capellan_dihedral}).
We also note that any proper birational model $X\to \bA^2/G$
dominated by $Y$ has only one $(-1)$-curve. In particular, the MMP of $Y$ to $\bC^2/G$ is unique in our case.

\subsection{McKay correspondence for dihedral reflection groups}
We use the notation in the previous section.
The quotient stack $[Y_H/(G/H)]$ is the $2$nd root stack of $Y$ along the smooth branch divisor $\pi(R)\subset Y$, and
 $R$ has one ($n$:odd) or two($n$:even) connected components isomorphic to $\bA^1$. 
 
 \begin{lemma}[{\cite[Theorem 5.1]{Collins-Polishchuk}}]Let $i:R\hookrightarrow Y$ be the inclusion and $\pi: Y_H\to Y$ be the quotient map. 
    There are semiorthogonal decompositions of $D_c([Y_H/\bZ_2])$ into $D_c(R)$ and $D_c(Y)$ given by the following triangles:
    \begin{equation}\label{eq:triangle_pi_*}
        \pi^*\pi_*^{\bZ_2}E\to E\to \sgn\otimes i_*(\bfL i^*E\otimes \sgn)^{\bZ_2}
    \end{equation}
\end{lemma}
\begin{lemma}\label{lem:DerivedMcKay}
    There is an exact equivalence \[\Phi: D_c([Y_H/\bZ_2])\stackrel{\simeq}{\longrightarrow}D_c([\bC^2/G]),\]
   such that $\Phi^{-1}$ takes the simple $\bZ_2$-equivariant sheaves
\[(\rho_i\oplus \rho_{n-i})\otimes \scO_0(i=1,\dots,\lfloor \frac{n}{2}\rfloor),\,\, \rho_{\frac{n}{2}}^+\otimes \scO_0, \rho_{\frac{n}{2}}^-\otimes \scO_0, \rho_0^+\otimes \scO_0, \rho_0^-\otimes \scO_0\]
to  
\[\scO_{E_i}(-1)\oplus \scO_{E_{n-i}}(-1), \scO_{E_{\frac{n}{2}}}(-1),\scO_{E_{\frac{n}{2}}}(-1)\otimes \sgn, \omega_E[1], \omega_E\otimes \sgn [1],\]
where $E:=\cup_i E_i$ be the exceptional locus of $Y_H$.
\end{lemma}
\begin{proof}
    The exact equivalence $\Phi^H:D(Y_H)\stackrel{\simeq}{\longrightarrow}D([\bC^2/H])$ induced by the universal $H$-cluster $\scZ$ takes a skyscraper sheaf $\scO_y$ to an $H$-cluster $\scZ_y$. As in \cite[Theorem 4.1]{IU}, $\Phi^H$ induces the  equivalence $\Phi:D([Y_H/\bZ_2])\stackrel{\simeq}{\longrightarrow} D([\bC^2/G])$. Since this equivalence preserves the property that a complex has compact support, the desired equivalence follows. The latter correspondence follows from the description of simple objects in ${}^{0}\Per(Y_H/(\bC^2/H))$ by Van den Bergh \cite[Proposition. 3.5.8]{VdB_NCCR} after choosing the $G/H$-equivariant structures on $\scO_{E_{\frac{n}{2}}}(-1),\omega_E[1]$.
\end{proof}

\subsection{Generalities on Bridgeland stability conditions}\label{sc:Generalities_Stab}
Let $X$ be a smooth variety or a smooth DM stack (mainly $[Y_H/\bZ_2]$ or $[\bC^2/G]$). We write
\[D_c(X):=D^{b}_c(\coh X),\]
for the bounded derived category of coherent sheaves with compact support. For the treatment of Bridgeland stability on non-compact varieties, we refer to \cite{Bayer-Craw-Zhang}.
\begin{definition}
     $K^{\mathrm{num}}_c(X)$ is defined to be the quotient of the  Grothendieck group $K(D_c(X))$ by the subgroup consisting of classes
    $E\in K(D_c(X))$ with $\chi(E,F)=0$ for any $F\in D( X)$, where $\chi(E,F)$ is the Euler pairing
\[\chi(E,F)=\sum_{i\in \bZ}(-1)^i\dim \Hom(E,F[i]).\]  Similarly, $K^{\mathrm{num}}(X)$ is defined to be the quotient of $K(D(X))$ by the subgroup of $E\in K(D(X))$ with $\chi(F,E)=0$ for any $F\in D_c(X)$.
\end{definition}
    By definition, the induced pairing $\chi:K^{\mathrm{num}}_c(X)\times K^{\mathrm{num}}(X)\to \bZ$ is nondegenerate. We also note that $K^{\mathrm{num}}_c(X)$ has finite rank for a smooth quasi-projective variety $X$ \cite[Lemma 5.1.1]{Bayer-Craw-Zhang}. 

\begin{definition}[{\cite{Bridgeland_Stab}}]
    A Bridgeland stability condition on $X$ is a pair 
    \[(Z,\scA),\]
    where $Z:K^{\mathrm{num}}_c(X)\to \bC$ is a group homomorphism and $\scA$ is the heart of a bounded t-structure on $D(X)$ such that the following holds:
    \begin{enumerate}
        \item (positivity) for any $E\in \scA\setminus \{0\}$, $Z(E)\in \bH:=\{r\exp(i\pi\phi)\in \bC\mid r>0,\phi\in (0,1]\}$.
        \item (Harder--Narasimhan property) for any $E\in \scA$, there is a filtration in $\scA$
        \[0=E_0\subset E_1\subset \dots\subset E_N=E\]
        such that each factor $F_i=E_i/E_{i-1}$ is $Z$-semistable with $\arg Z(F_i)>\arg Z(F_{i+1})$.
    \end{enumerate}
    Here, we say that $E\in \scA$ is $Z$-(semi)stable if for any $0\subsetneq F\subsetneq E$, we have $\arg Z(F)<(\leq) \arg Z(E)$.
\end{definition}
The homomorphism $Z$ is called a \emph{central charge}. The following lemma is useful to check the Harder-Narasimhan property:
\begin{lemma}[{\cite[Proposition 4.10]{Macri-Schmidt_lecture}}]\label{lem:criterion_HN}
Let $(Z,\scA)$ be a pair as above satisfying only the positivity condition. Assume that 
\begin{itemize}
    \item $\scA$ is noetherian.
    \item the image of $\Im Z$ is discrete in $\bR$.
\end{itemize}
Then, Harder-Narasimhan filtrations exist in $\scA$ with respect to $Z$.
\end{lemma}
\begin{definition}[\cite{kontsevich-Soibelman_supp}]
    A Bridgeland stability $(Z,\scA)$ on $X$ satisfies the support property if there is a norm $||-||$ on $K^{\mathrm{num}}_c(X)_\bR$ and a constant $C>0$ such that for any nonzero $Z$-semistable object $E\in \scA$, we have 
    \[||E||<C|Z(E)|.\]
\end{definition}
We denote by $\Stab(X)$ the set of Bridgeland stability conditions satisfying the support property. 
\begin{theorem}[\cite{Bridgeland_Stab}]\label{thm:deformation_thm}There is a natural topology on $\Stab(X)$ such that the forgetting map $\Pi:\Stab(X)\to K^{\mathrm{num}}_c(X)^{\vee}_\bC, (Z,\scA)\mapsto Z$ is a local homeomorphism. 
\end{theorem}

As in \cite{Bridgeland_Stab}, given a stability $(Z,\scA)$ and any $\phi\in (0,1]$, set
\[\scP(\phi):=\{E\in \scA\mid \text{$E$ is $Z$-semistable of phase $\phi$}\}.\]
More generally, if $\phi\in \bR$, take the unique $n\in \bZ$ such that $\phi-n\in (0,1]$ and set $\scP(\phi)$ to be $\scP(\phi-n)[n]$. Such a collection of subcategory $\{\scP(\phi)\}_{\phi\in \bR}$ is called a \emph{slicing}. For any interval $I\subset \bR$, we also write
\[\scP(I):=\la \scP(\phi)\mid \phi \in I\ra_{\mathrm{ex}}.\]
By the support property, one can show that $\scP(\phi)$ is of finite length, i.e., noetherian and artinian. So, by HN-property, $\scA=\scP(0,1]$ holds. Bridgeland showed that giving a stability condition $(Z,\scA)$ is equivalent to the data $(Z,\{\scP(\phi)\}_{\phi\in \bR})$ satisfying some conditions \cite[Proposition 5.3]{Bridgeland_Stab}. 

We recall that the natural right rotation action of the universal cover $\wt{\GL_2^{+}}(\bR)$ of $\GL_2^{+}(\bR)$ on $\Stab(X)$ induces a left action of the commutative subgroup $\bC$ on $\Stab(X)$, given by
\begin{align}\label{eq:def_action}
    R(w)\cdot (Z,\scP(\phi)):=(e^{\pi \sqrt{-1} w}Z,\scP(\phi-\Re w)).
\end{align}
In terms of hearts, we can write
\[R(w)\cdot(Z,\scP(0,1])=(e^{\pi \sqrt{-1}w}Z, \scP(-\Re w,-\Re w+1]).\]
We note that $n\in \bZ\subset \bC$ acts by shift $[n]$. The only case we explicitly use later is $w=\frac{1}{2}$, in which case, the rotated heart is the HRS tilt of the form $\scP(-\frac{1}{2},\frac{1}{2}]=\scP(0,\frac{1}{2}]*\scP(\frac{1}{2},1][-1]$ in the sense of \cite{Happel-Reiten-Smalo}. Also, we note that for any stability $\sigma\in \Stab(\scY)$ and $w\in \bC$, an object $E\in D_c(\scY)$ is $\sigma$-(semi)stable if and only if $E$ is $R(w)\cdot \sigma$-(semi)stable since the relative phases are preserved under the $\bC$-action.

\subsection{Bridgeland stability on smooth surfaces}
In this subsection, we adapt the standard construction of Bridgeland stability conditions on smooth projective surfaces to the compactly supported setting. Although this formulation for non-compact surfaces does not seem to be readily available in the literature, the arguments are essentially identical to the projective case.

Throughout this subsection, we assume that $f: X\to S$ is a smooth surface that is projective birational over an affine surface $S$. We define $\NS(X)_\bR$ to be the space of numerical divisor classes, namely, the quotient of $\Pic(X)_\bR$ by the classes $E$ satisfying $E\cdot C=0$ for any proper curve $C$ on $X$. We note that $K_c^{\mathrm{num}}(X)_\bR/[\scO_x]$ is generated by the classes
$[\scO_C]$ for proper curves $C\subset X$.
Hence, there is a natural embedding
\begin{align}\label{eq:NS_inside_K}
    \NS(X)_\bR\cong (K^{\mathrm{num}}_c(X)_\bR/[\scO_x])^\vee\hookrightarrow K^{\mathrm{num}}(X)_\bR,
\end{align}
where the first map is defined by $[L]\mapsto ([\scO_C]\mapsto \chi([L]-[\scO_X],\scO_C)=-L\cdot C)$, and the second map is also defined by the Euler pairing. The image is identified with the subspace of classes $[E]$ with $\chi(E,\scO_x)=0$. 

 Similarly to projective surfaces, for $(\omega,B)\in \NS(X)_\bC$, we define the central charge as follows:
\begin{align}
    Z_{\omega,B}(E)&:=-\int \ch(E)\cdot (1-B-\sqrt{-1}\omega)\label{eq:def_centralcharge1}\\
&=\ch_1(E)\cdot B-\ch_2(E)+\sqrt{-1}\ch_1(E)\cdot \omega.\label{eq:def_centralcharge2}
\end{align}
for $E\in K^{\mathrm{num}}_c(X)$. 
 Here, we clarify the meaning of the formula on a non-projective surface $X$.
Choose a smooth compactification $j: X\subset \overline{X}$ and choose an extension $\ol{\omega},\ol{B}$ of $\omega,B$. We then define 
\[-\int \ch(E)\cdot (1-B-\sqrt{-1}\omega):=-\int \ch(j_*E)\cdot (1-\ol{B}-\sqrt{-1}\ol{\omega}).\]
This is independent of the choice of the compactification and the extensions. 
Indeed, given two such compactifications $\ol{X}^i$ and extensions $\ol{\omega}^i,\ol{B}^i$ $i=1,2$, one may choose a third smooth projective compactification $\wt{X}$ dominating both, and compare the two expressions on $\wt{X}$. The same argument shows that the terms $\ch_2(E)$, $\ch_1(E)\cdot B$ and $\ch_1(E)\cdot \omega$ are also well-defined via compactifications. Throughout the paper, we use these expressions without further comment. We denote by
\[
\NS(X)_\bC:=\NS(X)_\bR\otimes \bC
\]
the parameter space of classes $B+\sqrt{-1}\omega$, and we often write such classes as pairs $(\omega,B)$.

Since we are interested in the moduli problems and the set of semistable objects does not change under the action of $\bC$ \pref{eq:def_action}, we concentrate on the subset $\Stab_{\mathrm{n}}(X)$ of $(Z,\scA)\in \Stab(X)$ with $Z(\scO_x)=-1$, which is a slice with respect to the $\bC$-action. We call a central charge (or a stability condition) satisfying $Z(\scO_x)=-1$ \emph{normalized}.
\begin{lemma}\label{lem:norm_central_charge}Let $X$ be a smooth surface that is projective birational over an affine surface. Then, the affine transformation over $\bR$
    \[Z_{*,*}:\NS(X)_\bC\to K^{\mathrm{num}}_c(X)^{\vee}_\bC,(\omega,B)\mapsto Z_{\omega,B}\]
    is an embedding whose image is precisely the space of normalized central charges.
\end{lemma}
\begin{proof}
 By definition, the map $Z_{*,*}$ is the direct sum of affine transformations over
 \begin{align}
     \Im Z_{*,0}:\NS(X)_\bR\to K^{\mathrm{num}}_c(X)_{\bR}^\vee,\,\,\,\, &\omega\to \ch_1(-)\cdot \omega,\label{eq:im_part_Z**}\\ 
     \Re Z_{0,*}:\NS(X)_\bR\to K^{\mathrm{num}}_c(X)_{\bR}^\vee, \,\,\,\,&B\to -\ch_2(-)+\ch_1(-)\cdot B\label{eq:re_part_Z**}
 \end{align}
 Via \pref{eq:NS_inside_K},
$\NS(X)_\bR$ is a real codimension one subspace of
$K^{\mathrm{num}}(X)_\bR$.

On the other hand,
the image of $\Im Z_{*,0}$
(resp. $\Re Z_{0,*}$)
is contained in the affine hyperplane
\[
W(\mathcal O_x)=0
\quad
(\text{resp. } W(\mathcal O_x)=-1).
\]
Thus, by the nondegeneracy of the pairing,
\eqref{eq:im_part_Z**} and
\eqref{eq:re_part_Z**}
are isomorphisms onto their images.
Consequently,
$Z_{*,*}$
is an isomorphism onto the space of normalized central charges.
\end{proof}
Thus, $\Stab_{\mathrm{n}}(X)$ fits in the following cartesian diagram
\begin{equation}\label{eq:cartesian_Stab_n}
    \begin{tikzcd}
    \Stab_{\mathrm{n}}(X)\arrow[r,hook]\arrow[d,"\Pi_{\mathrm{n}}"]&\Stab(X)\arrow[d,"\Pi"]\\
    \NS(X)_\bC\arrow[r,hook,"Z_{*,*}"]&K^{\mathrm{num}}_c(X)^{\vee}_{\bC}.
    \end{tikzcd}
\end{equation}
We first define the parameter space of central charges in $\NS(X)_\bC$. Set
\begin{align*}
    \Amp(X)&:=\{\text{ample numerical classes}\}\subset \NS(X)_\bR\\
\ol{\Amp}(X)&:=\Amp(X)\cup\bigcup_{f:X\to Y} f^*\Amp(Y)\subset \NS(X)_\bR,
\end{align*}
where $f$ runs over all contractions of a single $(-1)$-curve on $X$. To incorporate the $B$-field direction, for each blowdown $f:X\to Y$ with $(-1)$-curve $C$, we define
 \begin{equation}\label{eq:def_scB_f}
 \scB_f:=\{B\in \NS(X)_\bR\mid -1/2<B\cdot C<1/2\}.
 \end{equation} 
 We then define
 \begin{align}
    A(X)&:=\Amp(X)\times \NS(X)_\bR,\\
\ol{A}(X)&:=A(X)\cup\bigcup_{f:X\to Y} (f^*\Amp(Y)\times \scB_{f}).
\end{align}
 Next, for each $\omega\in \ol{\Amp}(X)$, we define a heart
 \begin{align*}
    \scA_\omega:=\begin{cases}
        \coh_c(X)&\text{if $\omega\in \Amp(X)$},\\
        {}^{-1}\Per_c(X/Y)&\text{if $\omega\in f^*\Amp(Y)$ for a single blowdown $f: X\to Y$.}
    \end{cases}
\end{align*}
Here, ${}^{-1}\Per_c(X/Y)$ is defined to be the HRS tilt $\scF_{-1}[1]*\scT_{-1}$\cite{Happel-Reiten-Smalo} with respect to the torsion pair $(\scT_{-1},\scF_{-1})$ on $\coh_c(X)$ given by 
\begin{align*}
	\scT_{-1}&:=\{\text{$E\in \coh_c(X)$}\mid \text{the canonical map $f^*f_*E\to E$ is surjective}\},\\
	\scF_{-1}&:=\{\text{$E\in \coh_c(X)$}\mid \text{$f_*E=0$}\}.
\end{align*}
This description is similar to that in \cite[\S 3]{VdB_NCCR}. Moreover, ${}^{-1}\Per_c(X/Y)$  is precisely the full subcategory of
${}^{-1}\Per(X/Y)$ in the sense of \cite[\S 3]{VdB_NCCR} consisting of objects whose cohomology sheaves are compactly supported.

The following proposition is known for projective surfaces. It is proved by Bridgeland \cite{Bridgeland_K3} for K3 surfaces, and by Arcara--Bertram \cite{Arcara-Bertram} for general smooth projective surfaces.
\begin{proposition}\label{pr:standardstability}
    Let $X$ be a smooth surface that is projective birational over an affine surface. For any rational point $(\omega,B)\in \ol{A}(X)_{\bQ}$, we have \[\sigma_{\omega,B}:=
        (Z_{\omega,B},\scA_\omega)\in \Stab_\mathrm{n}(X).\] 
\end{proposition}
We will provide a proof in \S 5 for clarity.
\begin{remark}
    Let $(\ol{\omega},\ol{B})$ be some extensions of $(\omega,B)\in A(X)$ to a smooth compactification $\ol{X}$, and $\coh^{\ol{\omega},\ol{B}} (\ol{X})$ be the associated HRS-tilt of $\coh (\ol{X})$ (see, for example, \cite[6.2]{Macri-Schmidt_lecture}). Although $\coh_c(X)$ is contained in $\coh^{\ol{\omega},\ol{B}} (\ol{X})$, the proposition above is not directly implied by the claim for projective surfaces. Indeed, there are sheaves $E\in \coh_c(X)$ that are $Z_{\omega,B}$-semistable but not $Z_{\ol{\omega},\ol{B}}$-semistable in $\coh^{\ol{\omega},\ol{B}} (\ol{X})$. Fortunately,  the proof for projective surfaces can be applied to our cases. We give details in \S 5.
\end{remark}

\begin{proposition}\label{pr:continuity_standard_stab}
    Let $X$ be a smooth surface that is projective birational over an affine surface. The map $\ol{A}(X)\to K^{\mathrm{num}}_c(X)_{\bC}^\vee, (\omega,B)\to Z_{\omega,B}$ lifts to a continuous map 
\[\sigma: \ol{A}(X)\to \Stab_\mathrm{n}(X)\]
which takes any rational point $(\omega,B)$ to the stability condition $\sigma_{\omega,B}$.
\end{proposition}
\begin{proof}
    Using the support property for rational points (\pref{pr:standardstability}), the proof in \cite[\S 5.1]{Toda_surfaceMMP} can be applied to our claim.
\end{proof}

\subsection{Moduli space of Bridgeland stable objects }
Let $\sigma=(Z,\scA) \in \Stab(\scY)$ and $v\in K_c^{\mathrm{num}}(\scY)$ be a primitive class. We consider the set
\[M_\sigma(v):=\{E\in \scA\mid \text{$E$ is $Z$-stable with $[E]=v$ in $K_c^{\mathrm{num}}(\scY)$}\}/\cong.\] 
The proof of the following statement is essentially contained in \cite[Proposition 9.3]{Bridgeland_K3}. See also \cite[Proposition 5.27]{Macri-Schmidt_lecture}:

\begin{proposition}\label{pr:wall&chamber_Stab}
Fix a primitive class $v\in K_c^{\mathrm{num}}(\scY)$. 
    There exists a locally finite collection
of actual walls
\[
\{\scW_i\}_{i\in I}
\]
in $\Stab(\scY)$ for objects of class $v$ such that the following hold:
\begin{enumerate}
    \item On each connected component $C$ of $\Stab(\scY)\setminus\bigcup_{i\in I}\scW_i$, for any $\sigma,\tau\in C$, $M_\sigma(v)=M_\tau(v)$ holds.
    \item For each $\sigma=(Z,\scP)\in \scW_i$, there is a destabilizing object. More precisely, there is a proper inclusion $F\subset E$ in $\scP(\phi)$ for some $\phi$ with class $v$.
\end{enumerate}
\end{proposition}

We are also interested in the geometric structure of $M_\sigma(v)$, not just as a set. We consider the following version of the moduli functor in \cite{Bayer-Craw-Zhang} by sending a separated scheme of finite type $S$ to a set
\[\scM_\sigma(v)(S):=\{\text{flat families of $Z$-stable objects in $\scA$ with class $v$ }\}/\sim.\]
Here, we define a flat family and its equivalence relation as follows: an object $E_S\in D(\scY\times S)$ is \emph{$S$-perfect} if locally on $S$, $E_S$ is quasi-isomorphic to a bounded complex of $q^{-1}\scO_S$-flat coherent sheaves for the projection $q:\scY\times S\to S$. We say an $S$-perfect complex $E_S$ is a \emph{flat family of $\sigma$-stable objects in $\scA$} if for any closed point $s\in S$,  $\bfL i_{s}E_S\in D(\scY)$ is a $\sigma$-stable object in $\scA$ for the inclusion $i_s:\{s\}\hookrightarrow S$. Moreover, for two families $E_S, E'_S$, we write  $E_S\sim E'_S$ if there is $L\in \Pic(S)$ such that $E_S\cong E'_S\otimes q^*L$. We discuss the representability of $\scM_\sigma(v)$ in our case in \pref{cr:existence_of_moduli}.

\section{Construction of geometric Bridgeland stability}
In this section, for each smooth surface $X$ admitting a birational morphism from the maximal resolution $Y$ over $\bC^2/G$, we construct a connected open subset $\scU(X)$ in a certain subspace $\Stab_{\mathrm{n}}(\scY)$ of the stability manifold (see \S 3.3). We establish most of \pref{mthm:A} at the end of this section. We will see that most arguments in \cite{Toda_surfaceMMP} carry over to our setting with only minor modifications. We therefore omit proofs when no modification is needed.

We write $\scY$ for the quotient stack $[Y_H/\bZ_2]$ and we regard $\coh \scY$ as $\coh^{\bZ_2}Y_H$. In this section, we often consider the composition $f\circ\pi$ of a morphism $f: Y\to X$ and the coarse moduli map $\pi: \scY\to Y$. To simplify the notation, we write $\wt{f}$ for $f\circ \pi$. By abuse of notation, throughout the paper, we write
\[
\scO_{\wh{C}}:=\wt{f}^*\scO_C
\]
for the total transform of a curve $C$. For a birational morphism $f: Y\to X$ between smooth surfaces, or from $\scY$ to a smooth surface, we frequently use the left adjoint of $\bfL f^*$ 
\[\bfR f_!:=\bfR f_*(-\otimes\omega_Y)\otimes \omega_X^{-1}.\]

\subsection{Construction of the heart \texorpdfstring{$\scC_{\scY/X}^0$}{C(Y/X,0)}}
For any contraction $f: Y\to X$, set
\[\scC_{\scY/X}:=\Ker (\bfR \wt{f}_!)\subset D_c(\scY).\]
In this subsection, we construct a t-structure on $\scC_{\scY/X}$ and its heart $\scC_{\scY/X}^0$ following  Toda's construction \cite[\S 3.1]{Toda_surfaceMMP}. The idea of construction is to regard $\pi:\scY\to Y$ as a ``blowup in codimension 1''. More precisely, we prove the following statement:
\begin{proposition}\label{pr:constructionofC}
    For any smooth contraction $f:Y\to X$, we can associate the heart of a bounded t-structure $\scC_{\scY/X}^0\subset \scC_{\scY/X}$ satisfying the following properties:
    \begin{enumerate}[label=\textup{(\roman*)}]
        \item There is a set $\scS\subset \scC_{\scY/X}^0$ such that $\scC_{\scY/X}^0=\la \scS\ra_{\mathrm{ex}}$ and the image of $\scS$ in $K_c^{\mathrm{num}}(\scY)$ is a finite set. 
        \item For any object $F\in \scC_{\scY/X}^0$, $\bfR \wt{f}_*F[1]$ is a 0-dimensional sheaf.
        \item $\scC_{\scY/X}^0$ is noetherian.
        \item Take a factorization \[f:Y \overset{g}{\to} X'\overset{h}{\to} X\]  where $h$ is the blowdown of a (-1)-curve $C$ on $X'$, and suppose $\scC_{\scY/X'}^0=\la\scS'\ra_{\mathrm{ex}}.$ Then
        \[\scC_{\scY/X}^0=\la \scS,\scO_{\wh{C}}[-1]\ra_{\mathrm{ex}},\]
        where $\scS$ is the set of the cocone $S$ of the universal morphism
        \[S\to \scO_{\wh{C}}\otimes \Ext^1(\scO_{\wh{C}},S')\to S'[1]\]
        for some $S'\in \scS'$.
    \end{enumerate}
\end{proposition}
\begin{remark}
    In \cite{Toda_surfaceMMP}, the analogous heart
    \(\scC_{X/Y}^0\) is generated by finitely many objects,
    which is stronger than our property \textup{(i)}.
    However, property \textup{(i)} is sufficient for our purposes.
    See also \pref{rm:existence_moduli}.
\end{remark}

The first step is straightforward: 
\begin{lemma}The shift of the standard t-structure on $\scC_{\scY/Y}=i_*D_c(R)$
\[\scC_{\scY/Y}^0:=i_*\coh_c R[-1]=\la \scO_p[-1]\mid p\in R\ra_{\mathrm{ex}}\]
satisfies the conditions \textup{(i), (ii), (iii)}.
\end{lemma}

For induction, we take a factorization 
\[ f: Y \overset{g}{\longrightarrow} X' \overset{h}{\longrightarrow} X \]
where $h$ is the blowdown of a $(-1)$-curve $C$ on $X'$ and we assume that we already have the bounded heart $\scC_{\scY/X'}^0$ satisfying conditions \textup{(i), (ii), (iii)}. Now, we have a recollement
\[\scC_{\scY/X'}\overset{\mathrm{inc}}{\longrightarrow} \scC_{\scY/X}\overset{\bR\wt{g}_!}{\longrightarrow} \scC_{X'/X},\]
where the right and left adjoints of $\bR\wt{g}_!$ are given by
$\bfL \wt{g}^*$ and $\wt{g}^!$, respectively. 
We note that $\la \scO_{\wh{C}}\ra_{\mathrm{ex}}=\scC_{X'/X}\cap \coh_c(X')\subset \scC_{X'/X}$ is the heart of a bounded t-structure (cf. \cite[Proposition 3.5.8]{VdB_NCCR}). We also have the heart $\scC_{\scY/X'}^0=\la\scS'\ra_{\mathrm{ex}}$ by the induction assumption.
Let $\wt\scC_{\scY/X}^0$ be the heart of the glued t-structure induced from the hearts $\scC_{\scY/X'}^0\subset \scC_{\scY/X'}$ and $\la \scO_{\wh{C}}\ra_{\mathrm{ex}}\subset \scC_{X'/X}$ in the sense of \cite[number 1.4]{Beilinson-Bernstein-Deligne}. 

\begin{lemma}\label{lem:torsionpair_for_tiltC}
There are torsion pairs
    \begin{align}
        \wt\scC_{\scY/X}^0&=\scC_{\scY/X'}^0*\la\scO_{\wh{C}}\ra_{\mathrm{ex}},\label{eq:gluingistorsionpair}\\
        \wt\scC_{\scY/X}^0&=\la\scO_{\wh{C}}\ra_{\mathrm{ex}}*\scO_{\wh{C}}^{\perp}.\label{eq:torsionpair_for_tiltC}
    \end{align}
\end{lemma}
\begin{proof}
        The proof of the former torsion pair is identical to that of \cite[Lemma 3.1]{Toda_surfaceMMP}.
    For the latter torsion pair, once the noetherian property of \(\wt\scC_{\scY/X}^0\) is established, the same argument as in \cite[Lemma 3.2]{Toda_surfaceMMP} applies. We show that $\wt\scC_{\scY/X}^0$ is noetherian. We note that $\scO_{\wh{C}} $ has no higher self-extensions since $\scO_C$ has no higher self-extensions. 
        Let $
E = E_1 \twoheadrightarrow E_2 \twoheadrightarrow \dots \twoheadrightarrow E_i \twoheadrightarrow \cdots
$
    be an infinite sequence in $\wt\scC_{\scY/X}^0$. By the torsion pair \pref{eq:gluingistorsionpair}, we have the following commutative diagram:
    \[
    \begin{tikzcd}
        0\arrow[r] & T_i\arrow[r]\arrow[d]& E_i\arrow[r]\arrow[d, two heads]&\scO_{\wh{C}}^{\oplus n_i}\arrow[r]\arrow[d]& 0\\
        0\arrow[r]& T_{i+1}\arrow[r] &E_{i+1}\arrow[r]& \scO_{\wh{C}}^{\oplus n_{i+1}}\arrow[r]& 0
    \end{tikzcd}
    \]
    where $T_i\in \scC_{\scY/X'}^0$. Then the morphism $\scO_{\wh{C}}^{\oplus n_i}\to \scO_{\wh{C}}^{\oplus n_{i+1}}$ is also surjective and is, moreover, an isomorphism for sufficiently large $i$. So, we only need to show that the infinite sequence $T_1\twoheadrightarrow\dots \twoheadrightarrow T_i\twoheadrightarrow T_{i+1}\twoheadrightarrow\cdots$ is eventually constant. This follows from the induction hypothesis that $\scC_{\scY/X'}^0$ is noetherian. 
\end{proof}
We define $\scC_{\scY/X}^0$ to be the heart $\scO_{\wh{C}}^{\perp}*\la\scO_{\wh{C}}\ra_{\mathrm{ex}}[-1]$ obtained by the HRS tilt associated to the torsion pair \pref{eq:torsionpair_for_tiltC}. We can verify that $\scC_{\scY/X}^0$ has the desired properties:
\begin{lemma}
\begin{enumerate}
    \item We have 
    \[\scC_{\scY/X}^0=\la \scO_{\wh{C}}[-1],S\mid S'\in \scS'\ra_{\mathrm{ex}} \]
    where $S$ is defined by the triangle
    \[S\to \scO_{\wh{C}}\otimes \Ext^1(\scO_{\wh{C}},S')\to S'[1].\]
    In particular, $\scC_{\scY/X}^0$ satisfies the condition \textup{(i)}.
    \item Set $\scS:= \{S\mid S'\in \scS'\}\cup\{\scO_{\wh{C}}[-1]\}$ as above. For any $S\in \scS$, $\bfR f_*S[1]$ is a 0-dimensional sheaf, i.e. $\scC_{\scY/X}^0$ satisfies the condition \textup{(ii)} for $\scS$. 
    \item $\scC_{\scY/X}^0$ is noetherian, i.e. $\scC_{\scY/X}^0$ satisfies the condition \textup{(iii)}. 
\end{enumerate}
\end{lemma}
\begin{proof}
    \begin{enumerate}
        \item This follows by the same argument as in \cite[Lemma 3.3]{Toda_surfaceMMP}.
        \item The proof is identical to that of \cite[Lemma 3.4.]{Toda_surfaceMMP}.
        \item The noetherian property can be proved by repeating the argument of \pref{lem:torsionpair_for_tiltC}.
    \end{enumerate}
\end{proof}

\subsection{Generator of \texorpdfstring{$\scC_{\scY/X}^0$}{C(Y/X,0)}}\label{sc:generator_C}
In this subsection, we describe an explicit extension generator of $\scC_{\scY/X}^0$ for any smooth contraction $f: Y\to X$.
We have a unique factorization of $f$
\[Y=X_1 \overset{g_1}{\longrightarrow} X_2
\overset{g_2}{\longrightarrow}\cdots
\overset{g_N}{\longrightarrow} X_{N+1}=X\]
where each $g_i$ is a contraction of a $(-1)$-curve $C_i\subset X_i$ (see the last part of \S 2.3).  Write $f_i:=g_{i-1}\circ\dots\circ g_{1}:Y\to X_{i}$. We note that for odd $n$, $C_1=\pi(E_{\frac{n-1}{2}}\cup E_{\frac{n+1}{2}}), C_2=\pi(E_{\frac{n-3}{2}}\cup E_{\frac{n+3}{2}}),\cdots ,C_N=\pi(E_{\frac{n-2N+1}{2}}\cup E_{\frac{n+2N-1}{2}})$ up to strict transform. For even $n$, the description is similar.

Let $p_i:=g_i(C_i)\in X_{i+1}$ for $i=1,\dots,N$ and $p_0:=R \cap \wh{C_{1}}=R_1\cap(E_{\frac{n-1}{2}}\cup E_{\frac{n+1}{2}})\subset Y$ for odd $n$, $p_0^1:=R_1\cap C_1,p_0^2:=R_2\cap C_1$ for even $n$. For $i=1,\dots,N-1$, we consider
the exact sequence
\[0\to \ol{S_i}\to g_i^*\scO_{C_{i+1}}\to \scO_{C_i}\to 0\]
and set $S_i:=\bfL \wt{f_i}^*\ol{S_i}=\wt{f_i}^*\ol{S_i}$. The sheaf $\ol{S_i}$ is isomorphic to $\scO_{\ol{g_i^*C_{i+1}-C_i}}(-p_i)$.

\begin{lemma}\label{lem:explicit_generator_C}
For odd $n$, an extension generator of $\scC_{\scY/X}^0$ is given by 

    \begin{align}
        \scO_p[-1](p\in R\setminus \{p_0\}), \scO_{\wh{C_1}}(-p_0), S_1,\dots,S_{N-1}, \scO_{\wh{C_N}}[-1].
    \end{align}

For even $n$, an extension generator of $\scC_{\scY/X}^0$ is given by 
    \begin{align}
        \scO_p[-1](p\in R\setminus \{p_0^1,p_0^2\}), \scO_{\wh{C_1}}(-p_0^1),\scO_{\wh{C_1}}(-p_0^2), S_1,\dots,S_{N-1},\scO_{\wh{C_N}}[-1].
    \end{align}
\end{lemma}
\begin{proof}
    We prove the lemma by induction on $N$. In both cases, $\scC_{\scY/Y}^0$ is generated by $\scO_p[-1],p\in R$. We first take the universal extension in $\wt\scC_{\scY/X_2}^0$ of $\scO_p[-1]$ by $\scO_{\wh{C_1}}$. Then, $\scO_p[-1]$ is invariant if $p$ is away from the compact curves. Suppose that $p\in R\cap C_1$. In the odd case, the universal extension is given by
    \begin{align}
        0\to \scO_{p_0}[-1]\to \scO_{\wh{C_1}}(-p_0)\to \scO_{\wh{C_1}}\to 0.
    \end{align}
    In the even case, it is
    \begin{align}
        0\to \scO_{p_0^i}[-1]\to \scO_{\wh{C_1}}(-p_0^i)\to \scO_{\wh{C_1}}\to 0.
    \end{align} So, the first step is done. 
    
    Suppose that the lemma holds for $f_{N-1}:Y\to X_{N}$. We claim that taking the universal extension by $\scO_{\wh{C_N}}$ does not change the generators except for $S_{N-1},f_{N-1}^*\scO_{\wh{C_{N-1}}}[-1]$. Indeed, $\scO_p[-1]$ is remains unchanged if $p$ is away from the compact curves, and we claim that \[\Ext^1(\scO_{\wh{C_N}},\scO_{\wh{C_1}}(-p_0))=0.\] We have an exact triangle in $D_c(X_2)$:
    \[\bfR \wt{f_{1}}_*\scO_{\wh{C_1}}(-p_0)\to \bfR \wt{f_{1}}_*\scO_{\wh{C_1}}\to \bfR \wt{f_{1}}_*\scO_{p_0}.\]
    Since the latter morphism $\scO_{p_0}\cong \bfR \wt{f_{1}}_*\scO_{\wh{C_1}}\to \bfR \wt{f_{1}}_*\scO_{p_0}\cong \scO_{p_0}$ is an isomorphism, $\bfR \wt{f_{1}}_*\scO_{\wh{C_1}}(-p_0)=0$ follows. By this vanishing and the adjointness, we get the desired vanishing of $\Ext^1$. Since the proof for $\scO_{\wh{C_1}}(-p_0^i)$ is similar, we omit the details. The rest is verified similarly to the proof of  \cite[Proposition 3.7]{Toda_surfaceMMP}.
\end{proof}

\subsection{Construction of central charge}
We define 
\[\NS_{\mathrm{orb}}(\scY)_\bR:=\NS(Y)_\bR\oplus K^{\mathrm{num}}(R)_\bR\]
as the orbifold analogue of $\NS(X)_\bR$. 
First, we explain the motivation for the definition.
By the semiorthogonal decomposition \pref{eq:triangle_pi_*}, we have an isomorphism
\begin{align}
    \psi:K^{\mathrm{num}}_c(\scY)&\stackrel{\simeq}{\longrightarrow}K^{\mathrm{num}}_c(Y)\oplus K_c^{\mathrm{num}}(R),\label{eq:SOD_Grothendieck_grp}\\
    E&\mapsto (\pi_*^{\bZ_2}E,(\bfL i^*(\sgn\otimes E))^{\bZ_2}).\notag
\end{align}
Indeed, the semiorthogonal decomposition induces a decomposition of Grothendieck groups, which descends to a surjection on numerical Grothendieck groups. Since both sides are free of the same rank (deduced from the finer semiorthogonal decomposition \cite[Example 3.5]{IshiiNimurareflection}, for example), it is an isomorphism. Via the Euler pairing, this allows us to regard $\NS_{\mathrm{orb}}(\scY)_\bR$ as the subspace of $K^{\mathrm{num}}(Y)_\bR\oplus K^{\mathrm{num}}(R)_\bR\cong K^{\mathrm{num}}(\scY)_\bR$ with $\chi(E,\pi^*\scO_y)=0$. This is analogous to \pref{eq:NS_inside_K}.
\begin{remark}
    A similar definition also appears in \cite[Definition 8.1]{Darda-Yasuda_Manin1} and $\NS_{\orb}(\scY)_\bR$ can be naturally viewed as a subset of the Chen-Ruan cohomology (orbifold cohomology) $H^*_{\mathrm{CR}}(\scY,\bR)$. It is known that the Chen-Ruan cohomology group can be regarded as the direct sum of the ordinary cohomology groups of connected components of the inertia stack. In our case, we have \[H^*_{\mathrm{CR}}(\scY,\bR)\cong H^*(Y,\bR)\oplus H^*(R,\bR).\]
    where the second summand corresponds to the twisted sectors of $\scY$. In this paper, we only use the ordinary intersection theory and will not discuss any intersection products on  $\mathrm{H}^*_{\mathrm{CR}}(\scY,\bR)$. 
\end{remark}

Next, we define a homomorphism $Z_{\omega,B}^{t,s}\in K^{\mathrm{num}}_c(\scY)^\vee_\bC$ for $((\omega,t),(B,s))\in \NS_{\orb}(\scY)_\bC:=\NS_{\orb}(\scY)_{\bR}\otimes \bC$ as follows. Since $K^{\mathrm{num}}_c(R)\cong \bZ^{\pi_0(R)}$ is generated by  skyscraper sheaves $\scO_{p_i}$ on each connected component ${R_i}$ of $R$, we define $Z^{t,s}\in K_c^{\mathrm{num}}(R)^{\vee}_\bC$ for $\{s_i+\sqrt{-1}t_i\}_i\in \bC^{\pi_0(R)}\cong K^{\mathrm{num}}(R)_\bC$ as follows:
\begin{align*}
    Z^{t,s}:K_c^{\mathrm{num}}(R)&\to \bC\\
    \sum_{[R_i]\in \pi_0(R)}a_i[\scO_{p_i}]&\mapsto \sum_{i}a_i(-\frac{1}{2}+s_i+\sqrt{-1}t_i).
\end{align*}
Using the above homomorphisms, for $((\omega,t),(B,s))\in \NS_{\orb}(\scY)_\bC$, we define 
\begin{align*}
    Z_{\omega,B}^{t,s}:=(Z_{\omega,B}+Z^{t,s})\circ \psi\in K^{\mathrm{num}}_c(\scY)^\vee_\bC.
\end{align*}
For example, for $E\in D_c(Y)$ and  a point $p_i\in R_i$, we can calculate the values as follows:
\begin{align}
    Z_{\omega,B}^{t,s}(\pi^*E)&=-\ch_2^B(E)+\sqrt{-1} \ch_1(E)\cdot \omega.\notag\\
    Z_{\omega,B}^{t,s}(\sgn\otimes \scO_{p_i})&=-\frac{1}{2}+s_i+\sqrt{-1}t_i,\label{eq:Z(sgn_otimes_skyscraper)}\\
    Z_{\omega,B}^{t,s}(\scO_{p_i})&=Z_{\omega,B}^{t,s}(\pi^*\scO_{\pi(p_i)})-Z_{\omega,B}^{t,s}(\sgn\otimes \scO_{p_i})\notag\\
    &=-\frac{1}{2}-s_i-\sqrt{-1}t_i.\label{eq:Z(skyscraper)}
\end{align}
Here, the third equality is derived from the following exact sequence in $\coh_c(\scY)$: 
\[0\to \sgn\otimes \scO_{p_i}\to \pi^*\scO_{\pi(p_i)}\to \scO_{p_i}\to 0.\]
We justify the normalization term $-1/2$ in the definition of $Z^{t,s}$ in the following sense:
\begin{lemma}\label{lem:Z_2-inv_central_charge}
    $(s,t)=(0,0)$ if and only if $Z_{\omega,B}^{t,s}$ is $G/H$-invariant, i.e. $Z_{\omega,B}^{t,s}=Z_{\omega,B}^{t,s}\circ (\sgn\otimes -)$.
\end{lemma}
\begin{proof}
    The “only if” direction follows immediately from \pref{eq:Z(skyscraper)} and \pref{eq:Z(sgn_otimes_skyscraper)}. Assume that $(s,t)=(0,0)$. It is enough to show that 
    \[Z_{\omega,B}^{0,0}=\frac{1}{2}\left(Z_{\omega,B}(\pi_*^{\bZ_2}(-))+Z_{\omega,B}(\pi_*^{\bZ_2}(\sgn\otimes-))\right)\] First, we claim that 
    \begin{align}\label{eq:class_identity}
        [\pi_*^{\bZ_2}E]=[\pi_*^{\bZ_2}(\sgn\otimes E)]+[i_*(\bfL i^*E)^{\bZ_2}] 
    \end{align}
    holds in $K^{\mathrm{num}}_c(Y)$ for any $E\in K^{\mathrm{num}}_c(\scY)$. Indeed, it is trivial on $K^{\mathrm{num}}_c(R)$, and on $K^{\mathrm{num}}_c(Y)$, it follows from the exact triangle
    \[F\otimes \scO_Y(-\pi(R))\to F\to \ol i_*\bfL \ol i^*F,\]
    where $\ol i: \pi(R)\to Y$ is the inclusion and $F\in K^{\mathrm{num}}_c(Y)$ is an arbitrary object.
    Then, for any $E\in D_c(\scY)$,  the following equality shows $\bZ_2$-invariance:
    \begin{align*}
        \frac{1}{2}\left(Z_{\omega,B}(\pi_*^{\bZ_2}E)+Z_{\omega,B}(\pi_*^{\bZ_2}(\sgn\otimes E))\right)
        &=Z_{\omega,B}(\pi_*^{\bZ_2}E)+ \frac{1}{2}Z_{\omega,B}(i_*(\bfL i^*(\sgn\otimes E))^{\bZ_2})\\
        &=Z_{\omega,B}(\pi_*^{\bZ_2}E)+Z^{0,0}(\bfL i^*(\sgn\otimes E)^{\bZ_2})\\
        &=Z_{\omega,B}^{0,0}(E).
    \end{align*}
    Here, we applied \pref{eq:class_identity} for $\sgn\otimes E$ in the first equality, and in the second equality we used the equality on $K^{\mathrm{num}}_c(R)$
    \[\frac{1}{2}Z_{\omega,B}(i_*-)=Z^{0,0}.\]
    
\end{proof}
\begin{lemma}\label{lem:char'n_of_image_NS_orb}
    The affine transformation over $\bR$
\[Z_{*,*}^{*,*}:\NS_{\orb}(\scY)_\bC\to K_{c}^{\mathrm{num}}(\scY)_\bC^\vee, (\omega,t,B,s)\mapsto Z_{\omega,B}^{t,s}\]
is an embedding whose image is precisely the space of central charges $Z$ satisfying $Z(\pi^*\scO_y)=-1$ for $y\in Y$.
\end{lemma}
\begin{proof}
    Similarly to \pref{lem:norm_central_charge}, $Z_{*,*}^{*,*}$ is the direct sum of the two affine transformations
    \begin{align}
     &\Im Z_{*,0}^{*,0}:\NS_{\orb}(\scY)_\bR\to K_{c}^{\mathrm{num}}(\scY)_{\bR}^\vee,\,\,\,\, \label{eq:im_part_orbZ**}\\ 
     &\Re Z_{0,*}^{0,*}:\NS_{\orb}(\scY)_\bR\to K_{c}^{\mathrm{num}}(\scY)_{\bR}^\vee,\label{eq:re_part_orbZ**}
 \end{align}
    whose image is the subspace of $W\in K_{c}^{\mathrm{num}}(\scY)_{\bR}^\vee$ with $W(\pi^*\scO_y)=0,-1$ respectively for any $y\in Y$.
\end{proof}
By \pref{lem:char'n_of_image_NS_orb}, we have the following cartesian diagram
\begin{equation}\label{eq:cartesian_Stab_n_orb}
    \begin{tikzcd}
        \Stab_{\mathrm{n}}(\scY)\arrow[r,hook]\arrow[d,"\Pi_{\mathrm{n}}"]&\Stab(\scY)\arrow[d,"\Pi"]\\
    \NS_{\orb}(\scY)_\bC\arrow[r,hook,"Z_{*,*}^{*,*}"]&K_c^{\mathrm{num}}(\scY)^{\vee}_{\bC},
    \end{tikzcd}
\end{equation}
where $\Stab_{\mathrm{n}}(\scY)$ is defined to be the subset of $(Z,\scA)\in \Stab(\scY)$ with $Z(\pi^*\scO_y)=-1$ for $y\in Y$. Similarly to the case of ordinary surfaces, we call such a stability condition \emph{normalized}, and likewise its central charge.

Let $f: Y\to X$ be a smooth contraction. Define
\[\NS_{\wt{f}}(\scY)_\bR:=(f^*\NS(X)_\bR)^{\perp}\oplus \bR^{\pi_0(R)}\subset \NS_{\orb}(\scY)_\bR,\]
where $(f^*\NS(X)_\bR)^{\perp}$ is the orthogonal space of $f^*\NS(X)_\bR$ in $\NS(Y)_\bR$ with respect to the intersection product.
 Then, $\NS_{\wt{f}}(\scY)_\bR$ is generated by the irreducible components of the exceptional locus of $f$ and the components of the twisted sectors. Formally, we regard $\bR^{\pi_0(R)}$ as $\NS_{\pi}(\scY)$ for the coarse moduli map $\pi:\scY\to Y$. For any $\omega\in \ol{\Amp}(X)$, we set
\[C_{\wt{f}}(X):=\{(D,t)\in \NS_{\wt{f}}(\scY)_\bR\mid \Im Z_{f^*\omega+D,0}^{t,0}(F)>0 \,\,\text{for all $F\in \scC_{\scY/X}^0$}\}.\]
\begin{lemma}\label{lem:cone_description}
    $C_{\wt{f}}(X)$ is a convex polyhedral cone that is independent of the choice of $\omega$. In particular, it is a connected, nonempty open subset of $\NS_{\wt{f}}(\scY)_\bR$.
\end{lemma}
\begin{proof}
    We can take a basis as follows
\[\NS_{\wt{f}}(\scY)_\bR=\bigoplus_{i=1}^{N}\bR[f_{i}^*C_i]\oplus \bigoplus_{[R_j]\in \pi_0(R)}\bR[R_j]\]
where $f_{i}:Y\to X_i$ denotes the contraction as in \S\ref{sc:generator_C}.
We write $(D,t)=\sum_i a_i[f_{i}^*C_i]+\sum_{[R_j]\in \pi_0(R)}t_j [R_j]$. 
Suppose that $n$ is odd.
By the description of generators in \pref{lem:explicit_generator_C}, we have 
\begin{align*}
    \Im Z_{f^*\omega+D,0}^{t,0}(\scO_p[-1])&=t,\\
    \Im Z_{f^*\omega+D,0}^{t,0}(\scO_{\wh{C_1}}(-p_0))&=t-a_1,\\
    \Im Z_{f^*\omega+D,0}^{t,0}(S_i)&=a_{i}-a_{i+1},\\
    \Im Z_{f^*\omega+D,0}^{t,0}(\scO_{C_N}[-1])&=a_N.
\end{align*}
Therefore, $C_{\wt{f}}(X)$ is identified with the following cone 
\[\{(a_1,\dots,a_N,t)\in \bR^N\oplus \bR\mid t>a_1>a_2>\dots>a_N>0\}.\]
When $n$ is even, similarly, one can see that
$C_{\wt{f}}(X)$ is identified with 
\[\{(a_1,\dots,a_N,t_1,t_2)\in \bR^N\oplus \bR^2\mid t_1>a_1>a_2>\dots>a_N>0, t_2>a_1\}.\]
\end{proof}
\begin{remark}
    Our $C_{\wt{f}}(X)$ is a linear cone, while the corresponding open set constructed by Toda in \cite{Toda_surfaceMMP}\S 4.1 also involves a quadratic cone. This difference is due to the absence of the term $\omega^2\ch_0$ in our central charge. Consequently, the comparison with the algebraic stability conditions in \S 4 is meaningful since the algebraic construction only produces linear cones.
\end{remark}
We define 
\begin{align*}
    A^{\dagger}(X)&:=\{(f^*\omega+D,t)\in \NS_{\orb}(\scY)_\bR\mid \omega\in \Amp(X), (D,t)\in C_{\wt{f}}(X)\},\\
    \ol A^{\dagger}(X)&:=\{(f^*\omega+D,t)\in \NS_{\orb}(\scY)_\bR\mid \omega\in \ol\Amp(X), (D,t)\in C_{\wt{f}}(X)\}.
\end{align*}
We have the natural projection
$\wt{f}_*:\ol A^{\dagger}(X)\to \ol \Amp(X)$ with constant fiber $C_{\wt{f}}(X)$. So, $A^{\dagger}(X)$ is a connected open subset of $\NS_{\orb}(\scY)_\bR$. 

Furthermore, using $\scB_f$ (see definition \pref{eq:def_scB_f}), we define
\[\scB:=\{(B,s)\in \NS_{\orb}(\scY)_\bR\mid f_*B\in \scB_f \,\,\text{for any smooth contraction $f:Y\to X$}\}.\] 
Then, $\scB$ is a convex open neighborhood of $0\in \NS_{\orb}(\scY)_\bR$. Indeed, $\scB$ is the intersection of finitely many open half-spaces. The finiteness follows from the fact that there are only finitely many compact curves on $Y$. In \S 3.5, we will regard $A^\dagger(X)\times \scB$ as the chamber corresponding to $X$.  

\subsection{Construction of the heart corresponding to \texorpdfstring{$\scU(X)$}{U(X)}}
For any smooth contraction $f:Y\to X$, we have an exact triple
\[\scC_{\scY/X}\stackrel{inc}{\to} D_c(\scY)\stackrel{\bfR \wt{f}_!}{\to} D_c(X).\]
By gluing the t-structures $\scC_{\scY/X}^0, \scA_\omega$ for $\omega\in \ol \Amp(X)$, we obtain a heart $\scA_{\omega}(\scY/X)$ on $D_c(\scY)$. 

\begin{lemma}\label{lem:property_A}
\begin{enumerate}
    \item We have a torsion pair $(\scC_{\scY/X}^0,\bfL \wt{f}^*\scA_\omega)$ of $\scA_{\omega}(\scY/X)$. In particular,
    \[\scA_{\omega}(\scY/X)=\scC_{\scY/X}^0*\bfL \wt{f}^*\scA_\omega.\]
    \item $\scA_{\omega}(\scY/X)$ is noetherian. 
    \item The subcategories $\scC_{\scY/X}^0, \bfL \wt{f}^*\scA_\omega$ of $\scA_{\omega}(\scY/X)$ are closed under quotients and subobjects.
\end{enumerate}
\end{lemma}

\begin{proof}
    \begin{enumerate}
        \item This follows from the same argument as in \cite[Lemma 3.9]{Toda_surfaceMMP}.
        \item Since \(\scC_{\scY/X}^0\) is noetherian by \pref{pr:constructionofC}, the argument of \cite[Lemma 4.3]{Toda_surfaceMMP} applies without modification.        
        \item The claim follows from the same argument as in \cite[Lemma 4.2]{Toda_surfaceMMP}.
    \end{enumerate}
\end{proof}

\subsection{Construction of \texorpdfstring{$\scU(X)$}{U(X)}}

\begin{lemma}\label{lem:stability_of_pullback}
    Let $(f^*\omega+D,t,B,s)\in \ol{A}^{\dagger}(X)\times \scB$. An object $M\in \scA_\omega$ is $Z_{\omega,f_*B}$-(semi)stable if and only if $\bfL \wt{f}^*M\in \scA_\omega(\scY/X)$ is $Z_{f^*\omega+D,B}^{t,s}$-(semi)stable. 
\end{lemma}
\begin{proof}
    This lemma immediately follows from \pref{lem:property_A}(3) and the fact that
    \[Z_{f^*\omega+D,B}^{t,s}(\bfL \wt{f}^*M)=Z_{\omega,f_*B}(M).\]
\end{proof}

\begin{lemma}\label{lem:stablity_is_constructed}
    For a rational point $(f^*\omega+D,t,B,s)\in \ol{A}^{\dagger}(X)\times \scB$, 
    \[\sigma_{f^*\omega+D, B}^{t,s}:=(Z_{f^*\omega+D,B}^{t,s},\scA_{\omega}(\scY/X))\]
    is a stability condition on $D_c(\scY)$.
\end{lemma}
\begin{proof}
    We first check the positivity. Let $E\in \scA_{\omega}(\scY/X)$. By \pref{lem:property_A}(1), we have an exact triangle
    \[T\to E\to \bfL \wt{f}^*M\]
    where $T\in \scC_{\scY/X}^0,M\in \scA_\omega$. By the definition of $C_{\wt{f}}(X)$, we have $\Im Z(T)>0$ if $T\neq 0$. So, it is enough to show the positivity for $\bfL f^*\scA_\omega$. We have 
    \[Z_{f^*\omega+D,B}^{t,s}(\bfL \wt{f}^*M)=Z_{f^*\omega+D,B}(\bfL \wt{f}^*M)=Z_{\omega,f_*B}(M).\]
    By \pref{pr:standardstability}, it follows that $Z_{\omega,f_*B}(M)\in \bH$.
    
    Next, we verify the Harder-Narasimhan property.  By \pref{lem:criterion_HN} and \pref{lem:property_A}(2), it remains to be proved that the image of $\Im Z_{f^*\omega+D,B}^{t,s}$ is discrete in $\bR$. This follows from the rationality of $f^*\omega+D,t$.
\end{proof}

\begin{proposition}\label{pr:support_property}
    For a rational point $(f^*\omega+D,t,B,s)\in \ol{A}^{\dagger}(X)\times \scB$, we have 
    \[\sigma_{f^*\omega+D,B}^{t,s}\in \Stab_{\mathrm{n}}(\scY),\]
    that is, $\sigma_{f^*\omega+D,B}^{t,s}$ satisfies the support property. 
\end{proposition}
\begin{proof}
    Similarly to \cite[Proposition. 4.7]{Toda_surfaceMMP}, the proof of the support property can be reduced to the support property of $\sigma_{\omega,f_*B}$ on the contraction $X$ (\pref{pr:standardstability}) by using \pref{pr:constructionofC}(1). 
\end{proof}
Moreover, we can prove the following:
\begin{proposition}\label{pr:conti_lift}
    The embedding $\ol{A}^{\dagger}(X)\times \scB\to \NS_{\orb}(\scY)_\bC$ lifts to a continuous map 
    \[\mu_X: \ol A^{\dagger}(X)\times \scB\to \Stab_{\mathrm{n}}(\scY)\]
    which takes any rational point $(f^*\omega+D,t,B,s)$ in $\ol A^{\dagger}(X)\times \scB$ to the stability condition $\sigma_{f^*\omega+D,B}^{t,s}$ in \pref{pr:support_property}.
\end{proposition}
\begin{proof}
   The proof is similar to \cite[Proposition 4.8]{Toda_surfaceMMP}.
\end{proof}
Set
\[\scU(X):=\mu_X(A^\dagger(X)\times \scB).\]
We note that $\Pi_{\mathrm{n}}$
induces a homeomorphism $\scU(X)\xrightarrow{\sim}
A^\dagger(X)\times\scB.$ In particular, $\scU(X)$ is a connected open subset in $\Stab_{\mathrm{n}}(\scY)$. 

\subsection{Boundary behavior of \texorpdfstring{$\scU(X)$}{U(X)}}
In this subsection, we study the relation of $\scU(X)$ under a single blowup. We also study a region contained in the boundary of $\scU(Y)$.
 
Suppose that a contraction $f:Y\to X$ decomposes as
\[ f: Y \overset{g}{\longrightarrow} X' \overset{h}{\longrightarrow} X \]
where $h$ is a blowdown of a $(-1)$-curve $C$ on $X'$. The first aim in this subsection is to prove that $\ol{\mu_X(A^\dagger(X))}\cap\ol{\mu_{X'}(A^\dagger(X'))}$ is nonempty and has real codimension one.
\begin{lemma}\label{lem:boundary_tilt}
    There are torsion pairs of the form:
    \begin{align}
        \scA_{h^*\omega}(\scY/X')&=\la \scO_{\wh{C}}\ra_{\mathrm{ex}} *\scO_{\wh{C}}^\perp,\label{eq:boundary_torsion_pair}\\
        \scA_\omega(\scY/X)&=\scO_{\wh{C}}^\perp*\la \scO_{\wh{C}}[-1]\ra_{\mathrm{ex}},
    \end{align}
    where $\scO_{\wh{C}}^\perp$ denotes the right orthogonal complement of $\scO_{\wh{C}}$ in $\scA_{h^*\omega}(\scY/X')$.
    In particular, $\scA_\omega(\scY/X)$ is the HRS tilt with respect to \pref{eq:boundary_torsion_pair}.
\end{lemma}
\begin{proof}
    The proof of \cite[Lemma 4.10, Lemma 4,11]{Toda_surfaceMMP} works using \pref{lem:property_A}.
\end{proof}
We note that $h^*\Amp(X)\subset \ol \Amp(X')$ is of codimension one.
We define  the subset $A_h^\dagger(X)\subset \ol A^\dagger(X')$ by the following cartesian diagram
\[
\begin{tikzcd}
    A_h^{\dagger}(X)\arrow[r,hook]\arrow[d,"\wt{g}_*"]&\ol A^\dagger(X')\arrow[d,"\wt{g}_*"]\\
    h^*\Amp(X)\arrow[r,hook]&\ol \Amp(X').
\end{tikzcd}
\]
By \pref{pr:conti_lift}, 
\begin{align}\label{eq:boundary_inc}
    \mu_{X'}(A_h^{\dagger}(X)\times \scB)\subset \ol{\scU(X')}
\end{align}
 is of codimension one.
\begin{proposition}\label{pr:connect_boundary}
    We have \[\mu_{X'}(A_h^{\dagger}(X)\times \scB)\subset \ol{\scU(X)}.\]
\end{proposition}
\begin{proof}
    The proof is the same as \cite[Proposition 4.12]{Toda_surfaceMMP}. We provide a proof here for the convenience of the reader. It is enough to prove the claim for rational points in $A_h^{\dagger}(X)\times \scB$. Let $(f^*\omega+D,t,B,s)$ be a rational point in $A_h^{\dagger}(X)\times \scB$. By \pref{eq:boundary_inc}, we have the point 
    \begin{align}
        \sigma_{f^*\omega+D,B}^{t,s}\in \ol{\scU(X')}.
    \end{align}
    On the other hand, for any sufficiently small rational number $a>0$, we have 
    \begin{align}
        (f^*\omega+D+aC,t,B,s)\in A^{\dagger}(X)
    \end{align}
    by the description of $C_{\wt{f}}$ in \pref{lem:cone_description}. So, we also have the point 
    \begin{align}
        \sigma_{f^*\omega+D+aC,B}^{t,s}\in \scU(X).
    \end{align}
It suffices to show that 
\begin{align}
    \lim_{a\to +0}\sigma_{f^*\omega+D+aC,B}^{t,s}=\sigma_{f^*\omega+D,B}^{t,s}
\end{align}
This is obvious at the level of central charges. Also, the hearts of $\sigma_{f^*\omega+D+aC,B}^{t,s}$ and $\sigma_{f^*\omega+D,B}^{t,s}$ are related by an HRS tilt by \pref{lem:boundary_tilt}. Then, by the following lemma, we get the desired result.
\end{proof}
\begin{lemma}[{\cite[Lemma 7.1]{Toda_curve_counting1}},{\cite[Lemma 4.13]{Toda_surfaceMMP}}]\label{lem:Toda_deformation_argument}
    Let $\scA,\scA'$ be the hearts of bounded t-structures on $\scD$, which is related by an HRS tilt. Let 
    \[[0,1)\ni t\mapsto Z_t\in K(\scD)_{\bC}^\vee\]
    be a continuous map such that $\sigma_t=(Z_t,\scA)$ for any rational number $t\in (0,1)$ and $\sigma_0$ determine points in $\Stab(\scD)$. Then we have $\lim_{t\to +0}\sigma_t=\sigma_0$.
 \end{lemma}

Set
\begin{align*}
    A^{\mathrm{tot}}&:=\bigcup_{f:Y\to X}\ol{A}^\dagger(X)\subset \NS_{\orb}(\scY)_\bR.
\end{align*}
We note that $A^{\mathrm{tot}}$ is a connected open subset of $\NS_{\orb}(\scY)_\bR$ by the description of $C_{\wt{f}}(X)$. 
\begin{corollary}\label{cr:total_conti_lift}
    The map
\[\mu_{\mathrm{geom}}: A^{\mathrm{tot}}\times \scB\to \Stab_{\mathrm{n}}(\scY)\]
defined by $\mu_{\mathrm{geom}}|_{\ol A^\dagger(X)\times \scB}=\mu_X$ is a continuous section of $\Pi_{\mathrm{n}}: \Stab_{\mathrm{n}}(\scY)\to \NS_{\orb}(\scY)_\bC$ \pref{eq:cartesian_Stab_n_orb} on the open subset $A^{\mathrm{tot}}\times \scB$.
\end{corollary} 
\begin{proof}
The continuity on each $A^\dagger(X)\times \scB$ is precisely \pref{pr:conti_lift}. Let $x=(\omega,t,B,s)\in A_h^\dagger(X)\times \scB$ for some blowdown $h:X'\to X$. By \pref{thm:deformation_thm}, we can take an  open neighborhood $U$ of $\mu_{X'}(x)$ on $\Stab_{\mathrm{n}}(\scY)$ such that $\Pi_{\mathrm{n}}|_U:U\to \Pi_{\mathrm{n}}(U)$ is a homeomorphism. By \pref{pr:connect_boundary}, $U\cap \mu_{X}(A^\dagger(X)\times \scB)\neq \emptyset$. By shrinking $U$, we may assume that $\Pi_{\mathrm{n}}(U)\cap (A^\dagger(X)\times \scB)$ is connected. We claim that
\[
\mu_X=(\Pi_{\mathrm{n}}|_U)^{-1} \quad\text{on}\quad  \Pi_{\mathrm{n}}(U)\cap (A^\dagger(X)\times \scB).
\]
Indeed, both maps are continuous sections of $\Pi_{\mathrm{n}}$, and they coincide at one point since
$U\cap \mu_X(A^\dagger(X)\times \scB)\neq \emptyset.$ Since $\Pi_{\mathrm{n}}|_U$ is a homeomorphism, the coincidence locus is open and closed in $\Pi_{\mathrm{n}}(U)\cap (A^\dagger(X)\times \scB)$, hence the claim follows from connectedness.

Similarly, after shrinking $U$ further if necessary, we have
\[
\mu_{X'}=(\Pi_{\mathrm{n}}|_{U})^{-1} \quad \text{on}\quad   \Pi_{\mathrm{n}}(U)\cap (\ol A^\dagger(X')\times \scB).
\] Therefore, $\mu$ coincides locally with $(\Pi_{\mathrm{n}}|_U)^{-1}$ near $x$, and hence is continuous at $x$.
\end{proof}

In the rest of the subsection, we prepare the lemmas to study the other direction of the boundary of $\ol{\mu_Y(A^{\dagger}(Y))}$, which will be used later. We note that a similar wall crossing phenomenon appears in the study of one-dimensional cases \cite{rota_thesis}.
Define 
      \begin{align*}
          \scT&:=\coh_c(R)=\la\scO_p\mid p\in R\ra_{\mathrm{ex}}\subset \coh_c(\scY),\\
          \scF&:=\scT^{\perp}\subset \coh_c(\scY).
      \end{align*}
\begin{lemma}
      $(\scT,\scF)$ is a torsion pair on $\coh_c(\scY)$.
\end{lemma}
\begin{proof}
    Since $\coh_c(\scY)$ is noetherian and $\scT$ is closed under quotients, the claim follows from \cite[Lemma 2.15.(i)]{Toda_curve_counting2}.
\end{proof}
\begin{lemma}\label{lem:tilt_of_standard_heart}We have 
    \[\scF*\scT[-1]=\scC_{\scY/Y}^0*\pi^*\coh_c(Y).\]
    In particular, $\scC_{\scY/Y}^0*\pi^*\coh_c(Y)$ is an HRS tilt of $\coh_c(\scY)$.
\end{lemma}
\begin{proof}
    Since both are the hearts of bounded t-structures, we only have to show that $\scC_{\scY/Y}^0*\pi^*\coh_c(Y)\subset \scF*\scT[-1]$. By definition, $\scC_{\scY/Y}^0=\scT[-1]$. We show that $\pi^*\coh_cY\subset \scF$. Let $E\in \coh_c(Y)$. Then we have
    \[\Hom(\scO_p,\pi^*E)=\Hom(\pi_!\scO_p,E)=0\]
    since $\pi_!\scO_p=0$.
\end{proof}

\subsection{Consequences for the wall-and-chamber structure}
In this subsection, we provide a proof of Theorem \pref{mthm:A}, except for the existence of the moduli space. 

\begin{proposition}\label{pr:moduli_reconstruction_sets}
    Let $f:Y\to X$ be a smooth contraction. 
    For any $\sigma\in \scU(X)$, as sets, we have 
    \[M_{\sigma}([\pi^*\scO_y])=\{\bfL \wt{f}^*\scO_x\mid x\in X\}.\]
     
\end{proposition}
\begin{proof}
    The proof is based on \cite[Proposition 4.14]{Toda_surfaceMMP}. By deforming $\sigma\in \scU(X)$, we may assume that $\sigma$ is of the form $\sigma_{f^*\omega+D,B}^{t,s}$ for a rational point $(f^*\omega+D,t,B,s)\in \scU(X)$. 
    Let $E\in \scA_\omega(\scY/X)$ be an element of $M_{\sigma}([\pi^*\scO_y])$. Then we have an exact triangle
    \[F\to E\to \bfL \wt{f}^*M.\]
    for some $F\in \scC_{\scY/X}^0, M\in \scA_{\omega}=\coh_cX$. If $F\neq 0$, by definition of $\scC_{\wt{f}}(X)$ and \pref{lem:stablity_is_constructed}, $\Im Z_{f^*\omega+D,B}^{t,s}(E)\geq \Im Z_{f^*\omega+D,B}^{t,s}(F)>0$. This contradicts $[E]=[\pi^*\scO_y]$. Then we have $E\cong \bfL \wt{f}^*M$. By \pref{lem:stability_of_pullback}, $M$ is $Z_{\omega,f_*B}$-stable of phase 1 and has the same numerical class as a skyscraper sheaf. So, $M$ must be isomorphic to $\scO_x$ for some $x\in X$. Conversely, for any $x\in X$, $\bfL \wt{f}^*\scO_x$ is $Z_{f^*\omega+D,B}^{t,s}$-stable of phase 1 by \pref{lem:stability_of_pullback} since $\scO_x$ is simple in $\coh_c(X)$.
\end{proof}
To prove a consequence on $M_{\sigma}([\pi^*\scO_y])$ as varieties, we need the existence of fine moduli space, which will be dealt in \pref{cr:existence_of_moduli}. The above proposition and the construction immediately imply the following lemma:
\begin{lemma}\label{lem:constancy_geom_chamber}
    The wall-and-chamber structure on $\Stab(\scY)$ in the sense of \pref{pr:wall&chamber_Stab} induces a wall-and-chamber structure on $A^{\mathrm{tot}}\times \scB$ with chamber $A^{\dagger}(X)\times \scB$. Moreover, for any $(B,s)\in \scB$, the same wall-and-chamber structure on $A^{\mathrm{tot}}\times \{(B,s)\}$ is obtained by pulling back the wall-and-chamber structure on $\Stab(\scY)$ along $\mu_{\mathrm{geom}}$.
\end{lemma}

\begin{remark}
    We can also construct an open chamber $A^{\dagger}(X)\times \NS_{\orb}(\scY)_\bR$ in $\NS_{\orb}(\scY)_\bC$, (or more strongly, in $\Stab_{\mathrm{n}}(\scY)$,)  where the moduli space $M_{\sigma}(v)$ is constant. However, some additional walls appear outside of $\scB$. For example, on the walls around the chamber $A(X)^{\dagger}\times \NS_{\orb}(\scY)_\bR$, ${}^{-1}\Per_c(X/Y)$ can be replaced by some twists of ${}^{-1}\Per_c(X/Y)$ by a line bundle (cf. \cite{Tramel-Xia_negative}, \cite[\S 4.2]{karube_NCMMP_blowup}). To avoid these additional walls, we put a restriction in $(B,s)$-direction.
\end{remark}

\section{Variations of Bridgeland stability and King stability}
In this section, we first recall the algebraic construction of Bridgeland stability due to Bayer--Craw--Zhang. After that, we provide the proof of \pref{mthm:B} by comparing the two constructions of Bridgeland stability conditions in \S 3 and \S 4.1. Using these results, we prove Ishii's conjecture in our setting. We also give an example illustrating the comparison of the corresponding wall-and-chamber structures.
\subsection{An algebraic region of the stability manifold}
In this subsection, we study some properties of stability conditions on $[\bC^2/G]$ constructed algebraically. To compare them with the normalized stability conditions on $\scY$ constructed in \S 3, we prepare the following parameter spaces:
\begin{align*}
    V_0&:=\{\theta\in\bR^{\Irr G}\mid \theta((\dim \rho)_\rho)=0\}\cong \Theta(G),\\
    V_1&:=\{\lambda\in \bR^{\Irr G}\mid \lambda((\dim \rho)_\rho)=1\}.
\end{align*}
For each $(\theta,\lambda)\in V_0\times V_1$, via the identification $K_c^{\mathrm{num}}([\bC^2/G])\cong \bigoplus_{\rho\in \Irr G}\bZ \rho\otimes \scO_0$, we consider the central charges of the form \begin{equation}\label{eq:def_of_Z_theta}
        Z_{\theta,\lambda}(M):=\theta(M)+\sqrt{-1}\cdot \lambda(M).
\end{equation}

\begin{lemma}
    The affine transformation
    \[Z_{*,*}:V_0\times V_1\to K^{\mathrm{num}}_c([\bC^2/G])_\bC^\vee, \,\,\,\,\,(\theta,\lambda)\mapsto Z_{\theta,\lambda}\]
    is an embedding whose image is precisely the subspace of the central charges $Z$ with  $Z(M)=\sqrt{-1}$ for any $G$-constellation $M$.
\end{lemma}
\begin{proof}
    As in \pref{lem:norm_central_charge}, $Z_{*,*}$ is the direct sum of the two affine transformations
    \begin{align}
        \Re Z_{*,0}:V_0\to K^{\mathrm{num}}_c([\bC^2/G])_{\bR}^\vee,\,\,\,\, &\theta\to \theta(\rho\otimes \scO_0),\label{eq:re_part_algZ**}\\ 
     \Im Z_{*,0}:V_1\to K^{\mathrm{num}}_c([\bC^2/G])_{\bR}^\vee, \,\,\,\,&\lambda\to \lambda(\rho\otimes \scO_0)\label{eq:im_part_algZ**}
    \end{align}
    whose images consist of the subspaces of $W\in K^{\mathrm{num}}_c([\bC^2/G])_{\bR}^\vee$ satisfying $W(M)=0$ and $W(M)=1$, respectively, for any $G$-constellation $M$.
\end{proof}

Set 
\[\Lambda_1:=\{\lambda\in V_1\mid \lambda_\rho>0\,\,\text{for all $\rho\in\Irr G$}\},\]
which is a convex open subset of $V_1$. The following proposition is due to Bayer--Craw--Zhang:
\begin{proposition}\cite[Lemma 7.1.3, Lemma 7.1.5, Proposition 7.1.6]{Bayer-Craw-Zhang}\label{pr:alg_stab}

    \begin{enumerate}
        \item $\sigma_{\theta,\lambda}:=(Z_{\theta,\lambda}, \coh_c([\bC^2/G]))\in \Stab([\bC^2/G])$.
        \item The map 
        \[\mu'_{\mathrm{alg}}: \Theta(G)\times \Lambda_1\to \Stab([\bC^2/G]), (\theta,\lambda)\mapsto \sigma_{\theta,\lambda}\]
is continuous. 
    \item For any $(\theta,\lambda)\in \Theta(G)\times \Lambda_1$, an object $E\in \coh_c([\bC^2/G])$ with $\theta(E)=0$ is $\sigma_{\theta,\lambda}$-(semi)stable if and only if it is $\theta$-(semi)stable.
    \item For any fixed $\lambda\in \Lambda_1$, the wall-and-chamber structure on $\Theta(G)$ is obtained by pulling back the wall-and-chamber structure on $\Stab([\bC^2/G])$ in \pref{pr:wall&chamber_Stab} with respect to $v=[\bC[G]\otimes \scO_0]$.
    \end{enumerate}
\end{proposition}
\begin{remark}
    The result by Bayer--Craw--Zhang \cite{Bayer-Craw-Zhang} is much stronger than \pref{pr:alg_stab}, which constructs an open subset of the whole stability manifold $\Stab([\bC^2/G])$. We restrict it to a smaller dimensional subspace by the condition $R(\frac{1}{2})\cdot Z_{\theta,\lambda,\xi}$ is a normalized central charge. See also \pref{lem:comparison_central_charge}.
\end{remark}

In the rest of this subsection, we prepare some lemmas to study the geometric interpretation of some stability conditions of the form $\sigma_{\theta,\lambda}$. We consider the following map
\[\Res^*:\Theta(H)\to \Theta(G), \theta\mapsto [\rho\mapsto \theta(\Res \rho)].\]
This is well-defined since $\Res \bC[G]\cong \bC[H]^{\oplus 2}$.
As in \pref{sc:rep}, we consider the action of $G/H$ on $\Theta(H)$ defined by conjugation and the action of the character group $\wh{G/H}$ on $\Theta(G)$ by tensoring $\triv, \sgn$. Set 
\begin{align*}
    \Theta(H)^{G/H}_+&:=\{\eta\in \Theta(H)^{G/H}\mid \eta_i>0\text{\,\,\,for all\,\,$i\neq 0$}\},\\
        \Theta(G)^{\wh{G/H}}_+&:=\{\theta\in \Theta(G)^{\wh{G/H}}\mid \theta_i>0\text{\,\,\, for all $i\neq 0_+,0_-$}\}.
\end{align*}

\begin{lemma}\label{lem:check_stable_for_comparison}
    
    \begin{enumerate}
        \item $\Res^*$ restricts to an isomorphism \[\Theta(H)^{G/H}\stackrel{\simeq}{\longrightarrow}\Theta(G)^{\wh{G/H}},\,\,\,\,\,\Theta(H)^{G/H}_+\stackrel{\simeq}{\longrightarrow}\Theta(G)^{\wh{G/H}}_+.\]
        \item Let $\theta\in \Theta(G)^{\wh{G/H}}_+$. Then 
    \[\Phi(\pi^*\scO_y) \text{ for all $y\in Y\setminus \pi(R)$} ,\,\,\,\,\, \Phi(\scO_p), \Phi(\scO_p\otimes \sgn) \text{ for all $p\in R$}\]
    are $\theta$-stable $G$-equivariant sheaves whose values of $\theta$ are $0$. In particular, for any $(\theta,\lambda)\in \Theta(G)^{\wh{G/H}}_+\times \Lambda_1$, these are $\sigma_{\theta,\lambda}$-stable objects of phase $\frac{1}{2}$. 
    \end{enumerate}
\end{lemma}
\begin{proof}
    \begin{enumerate}
        \item This is a standard consequence of the Clifford theory. Explicitly, for $\eta\in \Theta(H)$, we have
        \begin{align*}
            (\Res^*(\eta))_{i,n-i}&=\eta_i+\eta_{n-i},\,\,\,\text{for $i=1\dots,\lfloor \frac{n}{2} \rfloor-1$}\\
            \,\,(\Res^*(\eta))_{\frac{n}{2}}^{\pm}&=\eta_{\frac{n}{2}},\\
            \,\,(\Res^*(\eta))_{0}^{\pm}&=\eta_0.
        \end{align*}
        The claims follow immediately from this description. 
        \item By (1), $\theta=\Res^*\eta$ for some $0$-generated stability $\eta$. Since $\Phi(\scO_p)$ is an $H$-cluster, $\Phi(\scO_p)$ is $\eta$-stable. Similarly, $\Phi(\pi^*\scO_y)$ is of the form $N\oplus \beta^*N$, where $N$ is an $H$-cluster and $G/H=\la\beta\ra$. The action of $\beta$ is given by the permutation of the two summands. We note that $\Phi(\pi^*\scO_y)$ is $\eta$-semistable since it is the direct sum of $H$-clusters. Let $0\subsetneq E\subsetneq \Phi(\pi^*\scO_y)$ be a $G$-equivariant subsheaf with $\theta(E)=0$. Set $K:=E\cap \beta^*N\subset \beta^*N$ and $Q:=E/K\subset N$. Then, they fit into an exact sequence 
        \begin{align}\label{eq:exact_seq_constellation}
            0\to K\to E\to  Q\to 0.
        \end{align}
        By the $\eta$-semistability and $\theta(E)=0$, $\theta(K)=\theta(Q)=0$ follows.
        Since $N,\beta^*N$ are $\eta$-stable, the only possibility is $K=0,\beta^*N$ and  $Q=0,N$. In particular, by \pref{eq:exact_seq_constellation}, either $E=N$ or $E=\beta^*N$ follows. Since both $N, \beta^*N$ are not closed under the action of $\beta$, it contradicts the assumption that $E$ is a $G$-equivariant sheaf.

        The latter statement follows from \pref{pr:alg_stab}(3).
    \end{enumerate}
\end{proof}
\begin{remark}
    The degenerate stability $\theta\in \Theta(G)^{\wh{G/H}}_+$ appears in the construction of $G/H\hilb(H\hilb(\bC^2))$ \cite{Ishii-Ito-Nolla}. They constructed the desired stability by deforming this $\theta\in \Theta(G)^{\wh{G/H}}_+$ slightly to a generic direction. In our case, $G/H\hilb(H\hilb(\bC^2))$ is isomorphic to $Y$.
\end{remark}

\subsection{Comparison of the two regions}
In this subsection, we compare the geometric and algebraic regions in the stability manifold constructed in \S 3 and \S 4.1. In particular, we show that the corresponding local sections glue and induce the same wall-and-chamber structure. We begin by comparing their central charges:
\begin{lemma}\label{lem:comparison_central_charge}
    There are isomorphisms of affine spaces defined by
    \begin{align*}
        \gamma:\NS_{\orb}(\scY)_\bR&\stackrel{\simeq}{\to}V_0,&
        (\omega,t)\mapsto \left(\Im Z_{\omega,0}^{t,0}(\Phi^{-1}(\rho\otimes \scO_0))\right)_\rho,\\
        \delta:\NS_{\orb}(\scY)_\bR&\stackrel{\simeq}{\to}V_1,&
        (B,s)\mapsto \left(-\Re Z_{0,B}^{0,s}(\Phi^{-1}(\rho\otimes \scO_0))\right)_\rho.
    \end{align*}
\end{lemma}
\begin{proof}
    By \pref{eq:im_part_orbZ**},\pref{eq:re_part_algZ**}, $\gamma$ is defined so that the following diagram commutes
    \[
    \begin{tikzcd}
        \NS_{\orb}(\scY)_\bR\arrow[r,"\cong "',"\Im Z_{*,0}^{*,0}"]\arrow[d,"\gamma"]&\{W\mid W(\pi^*\scO_y)=0\}\arrow[r,hook]\arrow[d,"-\circ \Phi^{-1}","\cong "']&{K^{\mathrm{num}}_c(\scY)_\bC^\vee}\arrow[d,"-\circ \Phi^{-1}","\cong "']\\
        V_0\arrow[r,"\cong "',"\Re Z_{0,*}"]&\{W'\mid W'(\bC[G]\otimes \scO_0)=0
        \}\arrow[r,hook]&{K^{\mathrm{num}}_c([\bC^2/G])_\bC^\vee}
    \end{tikzcd}
    \]
    Similarly, by \pref{eq:re_part_orbZ**},\pref{eq:im_part_algZ**}, $\delta$ is defined so that the following diagram commutes
    \[
    \begin{tikzcd}
        \NS_{\orb}(\scY)_\bR\arrow[r,"\cong "',"\Re Z_{0,*}^{0,*}"]\arrow[d,"\delta"]&\{W\mid W(\pi^*\scO_y)=-1\}\arrow[r,hook]\arrow[d,"-1\circ(-)\circ \Phi^{-1}","\cong "']&{K^{\mathrm{num}}_c(\scY)_\bC^\vee}\arrow[d,"-1\circ (-)\circ \Phi^{-1}","\cong "']\\
        V_1\arrow[r,"\cong "',"\Im Z_{0,*}"]&\{W'\mid W'(\bC[G]\otimes \scO_0)=1
        \}\arrow[r,hook]&{K^{\mathrm{num}}_c([\bC^2/G])_\bC^\vee}
    \end{tikzcd}
    \]
\end{proof}
By the arguments above, we have the following commutative diagram:
\[
\begin{tikzcd}
    A^{\mathrm{tot}}\times \scB\arrow[r,hook]&\NS_{\orb}(\scY)_\bR\times \NS_{\orb}(\scY)_\bR\arrow[d,"\gamma\times \delta","\cong"']\arrow[r,"Z_{*,*}^{*,*}"]&{K^{\mathrm{num}}_c(\scY)_\bC^\vee}\arrow[d,"R(-\frac{\pi}{2})\circ(-)\circ \Phi^{-1}","\cong "']\\
    \Theta(G)\times \Lambda_1\arrow[r,hook]&V_0\times V_1\arrow[r,"{Z_{*,*}}"]&{K^{\mathrm{num}}_c([\bC^2/G])_\bC^\vee}
\end{tikzcd}
\]
    where $R(-\frac{1}{2})$ is the action of $\bC$ \pref{eq:def_action}. The equivalence $\Phi$ induces the following identification
    \[\Phi_*: \Stab(\scY)\overset{\sim}{\to}\Stab([\bC^2/G]), (Z,\scA)\mapsto (Z\circ \Phi^{-1}, \Phi\scA)\]
    
    We note that $R(\frac{1}{2})\circ\Phi^{-1}_*$ commutes with the forgetting maps $\Stab(\scY)\to K^{\mathrm{num}}_c(\scY)_\bC^\vee$ and $\Stab([\bC^2/G])\to K^{\mathrm{num}}_c([\bC^2/G])_\bC^\vee$. For the comparison with $\mu_{\mathrm{geom}}$, set
    \[\mu_{\mathrm{alg}}:=R(\frac{1}{2})\circ \Phi^{-1}_*\circ\mu'_{\mathrm{alg}}\circ(\gamma\times \delta): \NS_{\orb}(\scY)_\bR\times \delta^{-1}(\Lambda_1)\to \Stab_{\mathrm{n}}(\scY).\]
    \begin{lemma}\label{lem:property_gamma}
    $\gamma$ restricts to the isomorphism $\NS(Y)_\bR\stackrel{\simeq}{\longrightarrow}\Theta(G)^{\wh{G/H}}$ and moreover, $\Amp(Y)\stackrel{\simeq}{\longrightarrow}\Theta(G)^{\wh{G/H}}_+$.

    \end{lemma}
    \begin{proof}
The former statement follows from  \pref{lem:Z_2-inv_central_charge}. For the latter statement, we need more detailed expressions.
    Note that $\scY\to Y$ is generically an isomorphism along $\pi(E)$ and that $\sgn\otimes F$ and $F$ are isomorphic outside $R$ for $F\in D(\scY)$. For $\omega\in \NS(Y)_\bR$, $\gamma(\omega,0)\in \Theta(G)$ can be calculated as follows via \pref{lem:DerivedMcKay}:
        \begin{align*}
            \theta_{i,n-i}&=\omega\cdot \pi(E_i) \,\,\,\text{for $i=1\dots,\lfloor \frac{n}{2} \rfloor-1$}\\
            \theta_{\frac{n}{2}}^{\pm}&=\omega\cdot \pi(E_\frac{n}{2})\\
            \theta_0^{\pm}&=-\omega\cdot \pi(E),
        \end{align*}
        where $E$ is the exceptional locus. 
        So, the latter statement is immediate from the above description.
    \end{proof}
    
Now we can show that some rotated algebraic stability conditions can be interpreted geometrically:
\begin{lemma}\label{lem:invariant_alg_stab_are_geom}
    Let $(\theta,\lambda)\in \Theta(G)^{\wh{G/H}}_+\times \Lambda_{1}$. Then we have
    \[\left(R(\frac{\pi}{2})\circ \Phi^{-1}_*\right)(\sigma_{\theta,\lambda})=(Z_{\gamma^{-1}(\theta),\delta^{-1}(\lambda) },\coh_c(\scY)).\]
    In particular, we obtain a continuous map $\mu_{\mathrm{alg}}|_{\Amp(Y)\times \delta^{-1}(\Lambda_1)}: $
    \begin{align}
        \Amp(Y)\times \delta^{-1}(\Lambda_1)&\to \Stab_{\mathrm{n}}(\scY),\,\,\label{eq:mu_scY}\\
        (\omega,0,B,s)&\mapsto (Z_{\omega,B}^{0,s},\coh_c(\scY))=:\sigma_{\omega,B}^{0,s}\notag
    \end{align}
\end{lemma}
\begin{proof}
    By the definition of $\gamma,\delta$, the two central charges coincide. In order to compare the hearts, we only need the following lemma in combination with \pref{lem:check_stable_for_comparison}. 
    
    For the latter statement, we used \pref{lem:property_gamma}(1) and the continuity of $\mu'_{\mathrm{alg}}$ (\pref{pr:alg_stab}). 
\end{proof}

\begin{lemma}\cite[Lemma 10.1]{Bridgeland_K3}\cite[Lemma 5.4.2]{liu-Shen_root_stack_stab}
    Suppose $\sigma=(Z,\scA)\in \Stab(\scY)$ satisfies the following conditions:
    \begin{enumerate}
        \item $\pi^*\scO_y $ is a $Z$-stable object of phase $1$ for $y\in Y\setminus \pi(R)$.
        \item $\scO_p, \scO_p\otimes \sgn$ are $Z$-stable objects of phase $1$ for $p\in R$.
    \end{enumerate}
    Then, we have $\scA=\coh_c\scY$.
\end{lemma}
\begin{proof}
    The proof is essentially the same as \cite[Lemma 10.1]{Bridgeland_K3}. It suffices to show that any stable object with phase $\phi\in (0,1]$ lies in $\coh_c(\scY)$. First, let $E\in \scA$ be a $Z$-stable object with phase $\phi\in (0,1)$. By the assumption and Serre duality, we have  $\Hom^i(E,S)=0$ for any $i<0,1<i$ and any simple object $S$ in $\coh \scY$. Then, as in \cite[Proposition 5.4]{Bridgeland-Maciocia}, $E$ is quasi-isomorphic to 
    \[E^{-1}\to E^0\]
    with locally free $\bZ_2$-sheaf $E^i$. Since $E$ has compact support, $\scH^{-1}(E)\subset E^{-1}$ must be zero. So, $E\in \coh_c(\scY)$. 

    Let $E$ be a stable object with phase 1. By the same argument above, $\Hom^i(E,S)$ vanishes for any $i\neq 1$ and any simple object $S$ in $\coh \scY$. Then $E\cong E^{-1}[1]$ for a locally free $\bZ_2$-sheaf $E^{-1}$. This is impossible since $E$ has compact cohomologies.
\end{proof}
Unfortunately, the geometric stability conditions in \pref{lem:invariant_alg_stab_are_geom} are not contained in the stability conditions we constructed in \S 3. In order to compare them, we need the following:
\begin{lemma}\label{lem:intersection_nonempty_conn}
    \[
\scB':=\scB\cap \delta^{-1}(\Lambda_1)
\]
is a nonempty convex open subset of $\NS_{\orb}(\scY)_\bR$. In particular, it is connected.
\end{lemma}
\begin{proof}
    By direct calculation with \pref{lem:DerivedMcKay} and \pref{lem:Z_2-inv_central_charge}, we have $\delta(0)=(0,\dots,0,\frac{1}{2},\frac{1}{2})\in V_1$,whose nonzero entries correspond only to $\rho=\rho_0^+,\rho_0^-$. For example, for $\rho_{\frac{n}{2}}^+$ with even $n$, we can show that the exact triangle \pref{eq:triangle_pi_*} takes the form
        \[\pi^*\scO_{\pi_*E_{\frac{n}{2}}}(i)\to \scO_{E_{\frac{n}{2}}}(-1)\to \scO_Z\otimes \sgn,\]
         where $i$ is some integer and $Z$ is a 0-dimensional scheme on $R$ of length $2i-1$. Indeed, since $\pi_*\scO_{E_\frac{n}{2}}(-1)$ is torsion free, the desired triangle must sit in $\coh_c\scY$. We omit the details of the calculation. Now, $\delta(0)$ lies on the boundary of $\Lambda_1$.
    Since $\delta(\scB)$ is an open neighborhood of $\delta(0)\in V_1$, $\delta(\scB)\cap \Lambda_1\neq \emptyset$. The convexity follows from the convexity of $\scB$ and $\Lambda_1$.
\end{proof}
\begin{remark}
    As shown in the proof, $\delta(0)$ is not contained in $ \Lambda_1$. This is why we need to deform $\sigma_{\omega,0}^{t,0}$ in the $(B,s)$-direction.
\end{remark}
By restricting the continuous map \pref{eq:mu_scY}, we consider  
\[\mu_\scY:=\mu_{\mathrm{alg}}|_{\Amp(Y)\times \scB'}.\]
Now, we can show that the image of $\mu_{\scY}$ lies on the boundary of $\scU(Y)$ constructed in \S 3:
\begin{lemma}\label{lem:smallcomparison_stability}
    Let $(\omega,0,B,s)\in \Amp(Y)\times \scB'$. Then $\sigma_{\omega,B}^{0,s}$ lies on the boundary of $\mu_{\mathrm{geom}}(A^\dagger(Y)\times \scB')$ in $\Stab_{\mathrm{n}}(\scY)$. In particular, we obtain a continuous map
    \[\overline\mu_{\mathrm{geom}}:(A^{\mathrm{tot}}\times \scB)\cup(\Amp(Y)\times \scB')\to \Stab_{\mathrm{n}}(\scY)\]
    such that $\overline\mu_{\mathrm{geom}}|_{A^{\mathrm{tot}}\times \scB}=\mu_{\mathrm{geom}}$ and $\overline\mu_{\mathrm{geom}}|_{\Amp(Y)\times \scB'}=\mu_\scY$ in \pref{eq:mu_scY}.
\end{lemma}
\begin{proof}
    The proof of the former claim is similar to \pref{pr:connect_boundary}. In other words, it is enough to show the claim for rational points, and then, together with the support property (\pref{lem:invariant_alg_stab_are_geom}) and the fact that $\coh_c(\scY)$ is an HRS tilt of $\scC_{\scY/Y}^0*\pi^*\coh_c(Y)$ (\pref{lem:tilt_of_standard_heart}), Toda's deformation argument (\pref{lem:Toda_deformation_argument}) applies. Hence $\sigma_{\omega,B}^{0,s}$ lies on the boundary of $\mu_{\mathrm{geom}}(A^\dagger(Y)\times\scB')$.

    It remains to be shown that the extension is continuous. Since both $\mu$ and $\mu_\scY$ are local sections of the local homeomorphism $\Pi_{\mathrm{n}}$, the same local-sheet argument as in \pref{cr:total_conti_lift} applies. Namely, near any point $(\omega,0,B,s)\in \Amp(Y)\times\scB'$, the boundary condition above implies that the two sections choose the same local branch of $\Pi_{\mathrm{n}}$. Therefore $\overline{\mu}_{\mathrm{geom}}$ locally coincides with the inverse of this branch, and is continuous.
\end{proof}
Now, we have two local sections of $\Pi_{\mathrm{n}}:\Stab_{\mathrm{n}}(\scY)\to  \NS_{\orb}(\scY)_\bC$ :
\begin{align*}
    \overline{\mu}_{\mathrm{geom}}&:(A^{\mathrm{tot}}\times \scB)\cup(\Amp(Y)\times \scB')\to \Stab_{\mathrm{n}}(\scY),\\
    \mu_{\mathrm{alg}}&:\NS_{\orb}(\scY)_\bR\times \delta^{-1}(\Lambda_1)\to \Stab_{\mathrm{n}}(\scY).
\end{align*}
\begin{theorem}\label{thm:two_slices_glue}
     $\overline{\mu}_{\mathrm{geom}}$ and $\mu_{\mathrm{alg}}$ glue.
\end{theorem}
\begin{proof}
    We show that the two local sections are the same on the overlap $(A^{\mathrm{tot}}\cup \Amp(Y))\times \scB'$. By \pref{thm:deformation_thm}, $\Pi$ is a local homeomorphism. Since the overlap $(A^{\mathrm{tot}}\cup \Amp(Y))\times \scB'$ is connected by \pref{lem:intersection_nonempty_conn}, by using the arguments in the proof of \pref{cr:total_conti_lift}, we only have to show that the images of the two sections $\overline{\mu}_{\mathrm{geom}},\mu_{\mathrm{alg}}$ on $(A^{\mathrm{tot}}\cup \Amp(Y))\times \scB'$ have a nonempty intersection. By \pref{lem:invariant_alg_stab_are_geom} and \pref{lem:smallcomparison_stability}, the images of $\Amp(Y)\times \scB'$ coincide.
\end{proof}

\begin{corollary}\label{cr:wall&chamber_induced}
    The wall-and-chamber structure of $A^{\mathrm{tot}}$ defined in \pref{lem:constancy_geom_chamber} is obtained by restricting the GIT wall-and-chamber structure of $\Theta(G)$ via the embedding $\gamma: A^{\mathrm{tot}}\subset \NS_{\orb}(\scY)_\bR\stackrel{\simeq}{\longrightarrow}\Theta(G)$. Moreover, we can associate a GIT chamber $V(X)$  to each contraction $Y\to X$ such that 
    \begin{enumerate}
        \item $V(X)$ contains $\gamma(A^{\dagger}(X))$.
        \item if $X\to X'$ is a contraction of a $(-1)$-curve, $\ol{V(X)}\cap\ol{V(X')}$ is real codimension one in $\Theta(G)$.
    \end{enumerate}
\end{corollary}
\begin{proof}
    By \pref{pr:alg_stab},  the GIT chamber structure on $\Theta(G)$ is the same as the wall-and-chamber structure obtained by pulling back that of $\Stab_{\mathrm{n}}(\scY)$ for any $\lambda\in \Lambda_1$ along $\mu_{\mathrm{alg}}\circ (\gamma\times \delta)^{-1}|_{\Theta(G)\times \{\lambda\}}$. By \pref{lem:constancy_geom_chamber}, the same is true for the wall-and-chamber structure on $A^{\mathrm{tot}}$ via $\mu_{\mathrm{geom}}$. If we take any $(B,s)\in \scB'$, by \pref{thm:two_slices_glue}, the following diagram commutes:
    \[
    \begin{tikzcd}
        A^{\mathrm{tot}}\times \{(B,s)\}\arrow[r,"\mu_{\mathrm{geom}}"]\arrow[d,"\gamma", hook]&\Stab_{\mathrm{n}}(\scY)\\
        \Theta(G)\times \{\delta(B,s)\}\arrow[ru, "\mu_{\mathrm{alg}}\circ \gamma^{-1}"']
    \end{tikzcd}
    \]
     So, the wall-and-chamber structure corresponds by pullback, and then the former statement is shown. For the latter statement, for each contraction $Y\to X$, $\gamma(A^{\dagger}(X))$ is contained in a GIT chamber $V(X)$ in $\Theta(G)$. The property on the single blowup follows from the property of $A^{\dagger}(X)$.
\end{proof}
\begin{example}

Let $n=3$, i.e., $G$ is a dihedral reflection group of order $6$. $H=G\cap \SL_2(\bC)$ is a cyclic group of order $3$. Then, $Y_H=H\hilb(\bC^2)$ has two $(-2)$-curves $E_1,E_2$, corresponding to the non-trivial representations of $H$. The action of $G/H\cong \bZ_2$ on $Y_H$  permutes $E_1$ and $E_2$. The ramification divisor $R$ on $Y_H$ is isomorphic to $\bA^1$ and passes through $E_1\cap E_2$. Let $C$ be the image of $R$ under the quotient map $\pi:Y_H\to Y$. For any $D\in \NS_{\orb}(\scY)_\bR$, we can write $[D]=t[R]+a[C]$ for $t,a\in \bR$. The wall-and-chamber structure of $A^{\mathrm{tot}}$ can be described as in Figure \ref{fig:1}. The dotted line and the solid lines are the inverse image of the lines in $\Theta(G)$ under $\gamma^{-1}$ (see the explanation below). We note that $\NS(Y)_\bR$ is the subspace of $\NS_{\orb}(\scY)_\bR$ defined by the condition $t=0$. 

On the other hand, by taking the coordinate on $\Theta(G)$ given by $\theta_{0}^-,\theta_{1,2}$, we can also describe the GIT wall-and-chamber structure on $\Theta(G)$ as in Figure \ref{fig:2}. This can be calculated using Nolla--Sekiya's description \cite[Theorem 6.4]{Nolla-Sekiya}, which is used in Capellan's proof of Ishii's conjecture \cite{capellan_dihedral}. In this
figure, the solid lines denote walls, while the dotted line denotes only
the axis and is not a wall.  We note that $\gamma(A^{\dagger}(\bC^2/G))$ is strictly contained in $V(\bC^2/G)$. In this description, $\Theta(G)^{\wh{G/H}}$ is the subspace defined by $\theta_{0}^+=\theta_{0}^-$,  or equivalently, $\theta_0^-+\theta_{1,2}=0$.

On $\Theta(G)$, there are two kinds of symmetries of the chamber structure as in \cite[Lemma 2.6]{Craw-Ishii}. Explicitly, one is induced by tensoring with a character $\rho\in \wh{G/H}$. In our case, the nontrivial action is given by \[\theta\mapsto \theta(-\otimes \sgn).\] In the above coordinates, this action is described by
\[-\otimes \sgn: \theta_{1,2}\mapsto \theta_{1,2}, \,\,\,\theta_{0}^-\mapsto \theta_0^+=-2\theta_{1,2}-\theta_0^-.\]
The other symmetry is induced by the dual:
\[\theta\mapsto -\theta((-)^\vee).\]
In this example, all irreducible representations are self-dual. As a result, we have an action of $\bZ_2\times \bZ_2$ on the set of GIT chambers. In our case, one can see that any GIT chamber belongs to the orbit of either $V(Y)$ or $V(\bC^2/G)$ under this action. In this sense, the chambers $\{V(X)\}$ determine all GIT chambers up to these symmetries.

\begin{figure}[h]
\begin{minipage}{16em}
\centering
\begin{tikzpicture}[scale=1]
\fill[gray, pattern= north west lines](0,0) --  (2,2) -- (0,2) --cycle;
%\fill[gray, pattern= north west lines](0,0) -- (2,0) -- (2,2) -- (0,2) --cycle;
%%%%%%%%%%%%%%%%%
\fill[lightgray, pattern = north east lines](0,0) -- (-2,0) -- (-2,2)--(0,2) --cycle;
%\fill[lightgray, pattern = horizontal lines](0,0) -- (-2,2) -- (0,2) --cycle;
%%%%%%%%%%%%%%%%%

\draw[line width = 2pt,white](0,0)to (2,2);
\draw[thick, dotted](-2,-2)to (2,2);
\draw(-2,-1)to(2,1);
\draw[->](0,0) to (2.2,0);
\draw[->](0,0) to (0,2.2);
\draw(0,0) to (-2,0);
\draw(0,0) to (0,-2);
\node(A) at (65:1.5){\contour{white}{$A^\dagger(\bC^2/G)$}};
\node(B) at (-1,1){\contour{white}{$A^{\dagger}(Y)$}};
\coordinate[label=right:\(a\)](x)at(2.1,0);
\coordinate[label=above:\(t\)](x)at(0,2.1);
\end{tikzpicture}
\caption{\\$\{A^\dagger(X)\}_X$ in $\NS_{\orb}(\scY)_\bR$}
\label{fig:1}
\end{minipage}
\begin{minipage}{16em}
\centering
\begin{tikzpicture}[scale=1]
\fill[gray, pattern= north west lines](0,0) -- (2,-1) -- (2,2) -- (0,2) --cycle;
%\fill[gray, pattern= north west lines](0,0) -- (2,0) -- (2,2) -- (0,2) --cycle;
%%%%%%%%%%%%%%%%%
\fill[lightgray, pattern = north east lines](0,0) -- (-2,2) -- (0,2) --cycle;
%\fill[lightgray, pattern = horizontal lines](0,0) -- (-2,2) -- (0,2) --cycle;
%%%%%%%%%%%%%%%%%
\draw[line width = 3pt,white](0,0)to (2,0);
\draw[->](2.1999,0)to (2.2,0);
\draw[dotted,thick](0,0) to (2.2,0);
\draw[->](0,0) to (0,2.2);
\draw[dotted,thick](0,0) to (-2,0);
\draw(0,0) to (0,-2);
\draw(-2,2)to (2,-2);
\draw(-2,1)to(2,-1);
\node(A) at (1,1){\contour{white}{$V(\bC^2/G)$}};
\node(B) at (115:1.5){\contour{white}{$V(Y)$}};
\coordinate[label=right:\(\theta_0^{-}\)](x)at(2.1,0);
\coordinate[label=above:\(\theta_{1,2}\)](x)at(0,2.1);
\end{tikzpicture}
\caption{\\$\{V(X)\}_X$ in $\Theta(G)$}\label{fig:2}
\end{minipage}
\end{figure}

\end{example}

\subsection{Bridgeland moduli and King moduli}
Finally, we prove Ishii's conjecture in our setting. 
Let
$f:Y\to X$ be a projective birational morphism over $\bC^2/G$.
First,
 we prove the existence of the Bridgeland moduli space for our stability conditions $\sigma_{\omega,B}^{t,s}$ via King moduli spaces, bypassing
\pref{cr:wall&chamber_induced}. 

\begin{corollary}\label{cr:existence_of_moduli}Let $\sigma\in \scU(X)=\mu_{\mathrm{geom}}(A^{\dagger}(X)\times \scB)$. Then the moduli functor $\scM_{\sigma}([\pi^*\scO_y])$ is represented by a quasi-projective variety $M_{\sigma}([\pi^*\scO_y])$.
 \end{corollary}
 \begin{proof}
     Using \pref{cr:wall&chamber_induced}, by deforming $\sigma$ inside a chamber in $\Stab(\scY)$, we may assume that $\sigma=\mu_{\mathrm{alg}}((\gamma\times \delta)^{-1}(\theta,\lambda))$ for some $(\theta,\lambda)\in \Theta(G)\times \Lambda_1$. Since the King moduli functor is representable, it remains to show that, for any separated scheme of finite type $S$, 
     \[\scM_{\sigma}(v)(S)\to \scM_\theta(v)(S),\,\,\, E_S\mapsto \Phi_S(E_S), \]
     is a bijection, where $\Phi_S:=\Phi\times \id_S:D(\scY\times S)\to D([\bC^2/G]\times S)$ is the relative Fourier-Mukai functor. This follows from the same argument as in \cite[Proposition 7.3.1]{Bayer-Craw-Zhang}.
 \end{proof}
 
 The following proposition completes the proof of \pref{mthm:A}:
 \begin{proposition}\label{pr:moduli_reconstruction_var}
     Let $\sigma\in \scU(X)$. Then, $M_{\sigma}([\pi^*\scO_y])$ is isomorphic to $X$.
 \end{proposition}
 \begin{proof}
     In \pref{pr:moduli_reconstruction_sets}, we already proved that the map
    \[X\to M_{\sigma}([\pi^*\scO_y])\]  sending $x$ to $\bfL \wt{f}^*\scO_x$ is a bijection at the level of closed points. This map can be upgraded to a morphism of varieties since $(\bfL \wt{f}\times \id_X)^*\Delta_X\in D(\scY\times X)$ is a flat family of $\sigma$-stable objects with the class $[\pi^*\scO_y]$ and $M_{\sigma}([\pi^*\scO_y])$ is the fine moduli space (\pref{cr:existence_of_moduli}).
 Since 
    $\bfL \wt{f}^*: D_c(X)\to D_c(\scY)$
    is fully faithful, the map induced on the tangent space is an isomorphism. Thus, the map is an isomorphism.
 \end{proof}
 
 \begin{remark}\label{rm:existence_moduli}
     Toda uses the finiteness of the generator to prove the existence of moduli (especially the openness in Inaba's moduli of simple complexes). This is not obvious in our case. Instead, we prove the existence problem by using the moduli of representations, which is easier to construct.
 \end{remark}

 The following corollary is first proved by Capellan:
 \begin{corollary}[\cite{capellan_dihedral}]\label{cr:verify_conj}
     For any projective birational morphism $X\to \bC^2/G$ dominated by the maximal resolution $Y$ of the pair $(\bC^2/G,B)$, there is a generic stability $\theta\in \Theta(G)$ such that $M_\theta(v)\cong X$.
 \end{corollary}
 \begin{proof}
     Let $X$ be such a contraction. By \pref{cr:wall&chamber_induced}, there is a GIT chamber $V(X)$ in $\Theta(G)$ containing $\gamma( A^{\dagger}(X))$. By choosing any $\lambda\in \delta(\scB')$ and any $\theta\in \gamma( A^{\dagger}(X))$, we have 
     \[M_{\theta}(v)\cong M_{\sigma_{\theta,\lambda}}(v)\cong M_{\sigma_{\gamma^{-1}(\theta),\delta^{-1}(\lambda)}}(v)\cong X. \]
     Here, we used \pref{pr:alg_stab} for the first isomorphism,  \pref{thm:two_slices_glue} for the second isomorphism, and \pref{pr:moduli_reconstruction_var} for the last isomorphism.
 \end{proof}
 
 \begin{remark}\label{rm:capellan_approach}
     Capellan's proof depends on a  $G$-equivariant embedding $\bC^2\hookrightarrow \bC^3$ and the study of the variation of GIT on the moduli space of $G$-constellations on $\bC^3$ due to \cite{Nolla-Sekiya}. It is easy to write down the equation of walls in $\Theta(G)$ using the description of wall-and-chamber structure in \cite{Nolla-Sekiya}. Our approach is very different from his and provides a more conceptual interpretation of this statement. 
 \end{remark}

\section{Appendix}
In this section, we provide a proof of \pref{pr:standardstability} for clarity. The proof is based on the proof of the corresponding statement for smooth projective surfaces. Many steps are simpler since all compactly supported objects satisfy $\ch_0=0$

\begin{condition}\label{cd:omegaB} We consider two conditions on $(\omega,B)\in \ol{A}(X)$:
\begin{enumerate}
    \item $(\omega, B)\in A(X)=\Amp(X)\times \NS (X)$.
    \item $(\omega,B)\in f^*\Amp(Y)\times  \scB_f$ for a contraction $f:X\to Y$ of a single $(-1)$-curve $C$.
\end{enumerate}
\end{condition}

\begin{lemma}Suppose that $(\omega,B)\in \ol{A}(X)$ consists of rational classes.
    Then, $\sigma_{\omega,B}$ is a Bridgeland stability condition.
\end{lemma}
\begin{proof} We divide the proof into two cases (1),(2):
    \begin{enumerate}
        \item We prove the positivity first. Let $0\neq E\in \coh_c(X)$. $\Im Z_{\omega,B}(E)=\ch_1(E)\cdot \omega\geq 0$ since $\omega$ is ample. Suppose $\Im Z_{\omega,B}(E)=0$. Then, $\ch_1(E)=0$ by the ampleness of $\omega$. So $E$ is a 0-dimensional sheaf. In the numerical Grothendieck group, we can write $[E]=a[\scO_x]$ for some $a\in {\bZ_{>0}}$. So $Z_{\omega,B}(\scO_x)=-1$ implies $Z_{\omega,B}(E)\in \bR_{<0}$. 

        The HN-property follows from the discreteness of $\Im Z_{\omega,B}=\omega\cdot \ch_1$ by applying \pref{lem:criterion_HN}.
        \item Firstly, we show the positivity. Note that for any $E\in {}^{-1}\Per_c(X/Y)$, $\ch_1(E)\cdot f^*\omega=\ch_1(\scH^0(E))\cdot f^*\omega$ holds since $f_*\scH^{-1}(E)=0$. Let $ 0\neq E\in{}^{-1}\Per_c(X/Y).$ $\Im Z_{f^*\omega,B}(E)=\ch_1(E)\cdot f^*\omega \geq 0$ follows from the nefness of $f^*\omega$. Suppose $\Im Z_{f^*\omega,B}(E)=0$. Then, $E$ is 0-dimensional outside $C$. Then, as in \cite[Lemma 3.12]{Toda_ext_cont}, it is enough to show that \[Z_{f^*\omega,B}(\scO_C), Z_{f^*\omega,B}(\scO_C(-1)[1])\in \bH.\] This is equivalent to $-\frac{1}{2}<B\cdot C<\frac{1}{2}$.

        The HN-property follows similarly to \cite[Lemma 5.2]{Toda_ext_cont} (our situation is easier since we need fewer HRS-tilts from $\coh_c(X)$).
    \end{enumerate}
\end{proof}
    
\begin{proposition}\label{pr:quadratic_form_(1)}
    Assume \pref{cd:omegaB}(1). There is a constant $C_\omega>0$ that depends only on the class $\bP(\NS(X)_\bR)$ such that for any $\sigma_{\omega,B}$-semistable object $E\in \coh_c(X)$ we have
        \[\ch_1(E)^2+C_\omega\frac{(\ch_1(E)\cdot \omega)^2}{\omega^2}\geq 0.\]
\end{proposition}
\begin{proof}
    This is proved in \cite[Corollary 7.3.3]{Bayer-Macri-Toda_BG1}. We note that $\ch_1(E)$ is an effective divisor for $E\in \coh_c(X)$. Fix a norm $||-||$ on $\NS(X)$. There is a constant $A_\omega>0$ such that $|D^2\cdot\omega^2|\leq A_\omega||D||^2$ holds for any $D\in \NS(X)$. Also, there is a constant $B_\omega>0$ such that $D\cdot\omega\geq B_\omega ||D||$ holds for any effective divisor $D\in \NS(X)$. Set $C_\omega:=A_\omega/B_\omega^2$. Then, for any effective divisor $D$, we have\[ D^2\cdot\omega^2+C_\omega(D\cdot \omega)^2\geq D^2\cdot\omega^2+|D^2\cdot\omega^2|\geq 0\]
\end{proof}
We need an inequality in the case (2) analogous to \pref{pr:quadratic_form_(1)}. 
\begin{proposition}\label{pr:quadratic_form_(2)}
    Assume \pref{cd:omegaB}(2). There is a constant $C_\omega>0$ that depends only on the class $\bP(\NS(X)_\bR)$ such that for any $\sigma_{\omega,B}$-semistable object $E\in {}^{-1}\Per_c(X/Y)$ we have 
        \[\ch_1(E)^2+C_\omega\frac{(\ch_1(E)\cdot \omega)^2}{\omega^2}\geq -1.\]
\end{proposition}
\begin{proof}The proof is similar to \cite[Corollary 3.24]{Toda_ext_cont}. For the wall-crossing argument in \cite[Theorem 3.23]{Toda_ext_cont}, see also \cite[Theorem 3.5]{Bayer-Macri-Stellari}.
\end{proof}

Now, we are ready to prove the support property for rational points:

\begin{proof}(proof of \pref{pr:standardstability}.)
    Take a smooth compactification $j: X\hookrightarrow \ol{X}$. Then we have the embeddings
    \[K^{\mathrm{num}}_c(X)_\bR\hookrightarrow K^{\mathrm{num}}(\ol{X})_\bR\hookrightarrow H^*(\ol{X},\bR) \]
    via $j_*$ and the Chern character map. Now, the proof in \cite[3.7]{Toda_ext_cont} can be applied to our case (1), (2) by using \pref{pr:quadratic_form_(1)} and \pref{pr:quadratic_form_(2)}, respectively. The required estimates are simpler in our setting.
\end{proof}

\end{document}